\documentclass{article}
\usepackage[a4paper, left=3.0cm,right=3.0cm,top=3.85cm,bottom=3.85cm]{geometry}
\usepackage{amsmath,amsthm,amssymb}
\usepackage{graphicx} 
\usepackage{subcaption}
\DeclareMathOperator*{\argmin}{arg min}
\DeclareMathOperator{\Span}{span}
\usepackage[numbers]{natbib} 
\usepackage[hidelinks]{hyperref}     
\usepackage{doi}
\usepackage{authblk}
\date{}
\title{Symplecticity-Preserving Prediction of Hamiltonian Dynamics by Generalized Kernel Interpolation}
\author[1]{Robin Herkert \thanks{\texttt{robin.herkert@ians.uni-stuttgart.de}}}
\author[2]{Tobias Ehring \thanks{\texttt{tobias.ehring@ians.uni-stuttgart.de}}}
\author[3]{Bernard Haasdonk\thanks{\texttt{haasdonk@ians.uni-stuttgart.de}}}
\affil[1,2,3]{Institute of Applied Analysis and Numerical Simulation, University of Stuttgart, Pfaffenwaldring 57,
70569 Stuttgart, Baden-Württemberg, Germany}

\usepackage{hyperref}
\usepackage{cleveref}
\begin{document}

\maketitle
\newcommand{\p}{\mathbf{p}}
\newcommand{\q}{\mathbf{q}}
\renewcommand{\P}{\mathbf{P}}
\newcommand{\Q}{\mathbf{Q}}
\newcommand{\A}{\mathbf{A}}
\newcommand{\B}{\mathbf{B}}
\newcommand{\C}{\mathbf{C}}
\newcommand{\D}{\mathbf{D}}

\newcommand{\Ham}{\mathcal{H}}
\newcommand{\R}{\mathbb{R}}
\newcommand{\N}{\mathbb{N}}
\newcommand{\openQ}[1]{{\color{red} Q: #1}}
\newtheorem{theorem}{Theorem}
\newtheorem{lemma}{Lemma}
\newtheorem{corollary}{Corollary}
\newtheorem{definition}{Definition}
\newtheorem{remark}{Remark}
\newcommand{\Hk}{H_k(\Omega)}
\newcommand{\ip}[2]{\langle #1,#2\rangle_{\Hk}}
\newcommand{\norm}[1]{\left\lVert #1\right\rVert_{\Hk}}

\begin{abstract}
\noindent In this work, a kernel-based surrogate for integrating Hamiltonian dynamics that is symplectic
by construction and tailored to large prediction horizons is proposed. The method learns a scalar
potential whose gradient enters a symplectic-Euler update, yielding a discrete flow map that exactly preserves the
canonical symplectic structure. Training is formulated as a gradient Hermite--Birkhoff
interpolation problem in a reproducing kernel Hilbert space, providing a systematic
framework for existence, uniqueness, and error control.
 Algorithmically, the symplectic kernel
predictor is combined with structure-preserving model order reduction, enabling efficient treatment of high-dimensional discretized PDEs. Numerical
tests for a pendulum, a nonlinear spring--mass chain, and a semi-discrete wave equation
show nearly algebraic greedy convergence and long-time trajectory errors reduce by
two to three orders of magnitude compared to an implicit midpoint baseline at the same
macro time step.
\end{abstract}


\section{Introduction}
Many conservative physical phenomena, for instance in classical mechanics, theoretical chemistry, or molecular dynamics, can be formulated as Hamiltonian systems, whose mathematical structure encode the conservation of energy. In canonical coordinates $x = (q,p)^\top \in \R^{2n}, n \in \N$ and for $t \in I := [0,T] \subset \R$, the dynamics associated with a Hamiltonian function $\Ham \in C^{1}(\R^{2n},\R)$ are given by
\begin{align}\label{eqn:Hamsys}   
\dot x(t; x_0) = J_{2n} \nabla \Ham\left(x(t; x_0)\right), 
\qquad x(0; x_0) = x_0,
\end{align}
where
\[
J_{2n} := \begin{bmatrix} 0_n & I_n \\ -I_n & 0_n \end{bmatrix}
\]
is the canonical Poisson matrix, and $0_n, I_n \in \R^{n\times n}$ denote the zero and identity matrices of size $n$, respectively. For each initial value $x_0 \in \R^{2n}$, we denote by $x(t; x_0)$ the solution of the Hamiltonian system \eqref{eqn:Hamsys} at time $t \in I$.
As a standing example throughout this work, we consider for $n = 1$ the mathematical pendulum, which is governed by the Hamiltonian
\[
\Ham(q,p)
  = \frac{p^2}{2 m l^2} + m g l \left(1 - \cos q\right),
\]
and the corresponding equations of motion
\[
\dot q(t) = \frac{p(t)}{m l^2},
\qquad
\dot p(t) = - m g l \sin\left(q(t)\right),
\]
where $q$ denotes the angular displacement, $p$ the angular momentum, $m$ the mass, $l$ the pendulum length, and $g$ the gravitational acceleration.

\noindent Several properties follow from the Hamiltonian structure. In particular, the Hamiltonian is conserved along trajectories,
\[
\Ham\left(x(t; x_0)\right) = \Ham(x_0)
\quad\text{for all } t \ge 0,
\]
and the associated flow map $\Phi^{t} : \R^{2n} \to \R^{2n}$, defined by
\[
\Phi^{t}(x_0) := x(t; x_0),
\]
is symplectic. Here a differentiable map $\Psi : \R^{2n} \to \R^{2n}$ with Jacobian $D\Psi$ is called symplectic if
\[
\left(D\Psi(x)\right)^\top J_{2n}  D\Psi(x) = J_{2n}
\quad\text{for all } x \in \R^{2n}.
\]
\noindent
Note that, in the following, we omit the dependence of the solution $x(t; x_0)$ on the initial state $x_0$ whenever no confusion can arise and simply write $x(t)$.

\noindent
Symplectic maps enjoy several fundamental properties; see, e.g., \cite{daSilva2008}. In particular, $\det(D\Psi) = 1$, so that phase–space volume is
preserved and clouds of initial conditions are neither spuriously compressed
nor dilated. Symplectic maps are locally invertible, and the inverse of a
symplectomorphism (a diffeomorphism preserving the symplectic structure) is
again symplectic, which permits forward–backward evolution without loss of
structure. They are also closed under composition, so any finite sequence of symplectic maps
remains symplectic.
\noindent 
In practice, the flow map of a Hamiltonian system typically does not admit an explicit closed-form expression and must be approximated by a time discretization method. To preserve the structural properties of the exact symplectic flow in the discrete setting, one commonly employs symplectic integrators \cite{hairer2006}, which in applications often outperform general-purpose integrators that may lead to spurious energy gain or loss over long time intervals.
\\ 
One example of such an integrator is the symplectic Euler method,
\begin{align}\label{eq:symplEuler}
    x_{i+1} &= x_i + \Delta t  J_{2n} \nabla\Ham(q_i, p_{i+1}),
\end{align}
with $x_i = (q_i, p_i)^\top$, where the update map
\[
\Psi_\mathrm{Euler} : x_i \mapsto x_{i+1}
\]
is symplectic for every fixed $\Delta t > 0$ and every $\Ham \in C^{1}(\R^{2n},\R)$ (provided the implicit equation \eqref{eq:symplEuler} admits a solution), see \cite[Theorem~3.3]{hairer2006}. Furthermore, although symplectic integrators are typically implicit for general Hamiltonians, the symplectic Euler scheme becomes explicit when the Hamiltonian is separable, i.e., $\Ham(q,p) = T(p) + V(q)$.

\noindent Moreover, backward–error analysis shows that a symplectic integrator exactly preserves a modified Hamiltonian that is close to the original one; see \cite{hairer2006}. This explains the frequently observed bounded, typically oscillatory error in the Hamiltonian,
\[
e_\Ham(t) := \left|\Ham(x_0) - \Ham\left(x_\mathrm{approx}(t)\right)\right|, \qquad t \in \mathbb{T}_{\Delta t}:= \{k\Delta t| k \in \N_0\}
\]
over very long time horizons, where $x_\mathrm{approx}(t)$ denotes the approximate solution obtained by numerical time integration.

\noindent
The overall objective of this work is the simulation of Hamiltonian dynamics over long time horizons. However, symplectic time stepping can be computationally expensive, as stability and accuracy constraints often enforce small time steps. Therefore, we seek an approximation of the flow map that enables fast predictions for large time steps $\Delta T = K\Delta t  \gg \Delta t$. Kernel methods are particularly attractive
here because they yield an RKHS-based learning problem with a closed-form
solution (via Hermite--Birkhoff (HB) interpolation), provide direct access to derivatives
through reproducing identities, and come with rigorous approximation and convergence
theory, including greedy sparsification strategies and show good results in practice \cite{carlberg2019recovering,scholkopf2002learning,doeppel2024goal,deisenroth2013gaussian}. In the kernel setting, Hermite--Birkhoff (HB) interpolation, and more generally Hermite-type
kernel interpolation, is widely used across diverse applications, including surface
reconstruction \cite{zhong2019implicit}, PDE discretization \cite{la2008double}, image
reconstruction \cite{de2018image}, and optimal control \cite{ehring2024hermite}. In contrast,
we leverage HB interpolation to construct a prediction map that is symplectic by
construction.

\noindent In addition, the mapping defined by the kernel model should be symplectic in order to reflect the physical structure of the flow map and preserve its qualitative properties.
\noindent 
The central idea of our symplectic predictor is that, for a given $x_0 \in \R^{2n}$, we solve the implicit system
\begin{align}\label{eqn:pred}
    x_{\Delta T,\mathrm{pred}} &= x_0 + \Delta T  J_{2n} \nabla s\left(q_0, p_{\Delta T,\mathrm{pred}}\right),
\end{align}
and obtain a prediction $x_{\Delta T,\mathrm{pred}} = (q_{\Delta T,\mathrm{pred}}, p_{\Delta T,\mathrm{pred}})$, where $s : \R^{2n} \to \R$ is a learned, differentiable, scalar-valued surrogate kernel model. The resulting update map
\[
\Psi_s : x_0 \mapsto x_{\Delta T,\mathrm{pred}}
\]
is symplectic, since it reflects the symplectic Euler scheme applied to a Hamiltonian system with Hamiltonian $s$ and time step size $\Delta T$.

\noindent
To train our kernel model, we fix initial states 
\[
x_{0}^j = (q_{0}^j, p_{0}^j)^\top \in \R^{2n}, \qquad j = 1,\dots,M,
\]
and compute the corresponding time–$\Delta T$ propagated states (i.e., approximed solution after one macro time step, i.e.,  at $\Delta T$ using a symplectic integrator with small time step size $\Delta t$)
\[
x_{\Delta T}^j = (q_{\Delta T}^j, p_{\Delta T}^j)^\top
= \Phi^{\Delta T}\left(x_{0}^j\right) \qquad j = 1,\dots,M.
\]
Based on these data, we define the input–target pairs
\[
\xi_j := (q_{0}^j, p_{\Delta T}^j)^\top,
\qquad
y_j := J_{2n}^{\top}  \frac{x_{\Delta T}^j - x_{0}^j}{\Delta T},
\quad j = 1,\dots,M.
\]

\noindent The mixed argument $\xi_j = (q_{0}^j, p_{\Delta T}^j)^\top$ reflects the symplectic Euler mixed argument $(q_i, p_{i+1})$ in \eqref{eqn:pred}; with this choice, the identity
\[
\nabla s(\xi_j) = J_{2n}^{\top} \frac{x_{\Delta T}^j - x_{0}^j}{\Delta T} = y_j \qquad j = 1,\dots,M
\]
follows directly from \eqref{eqn:pred}. Interpreting $y_j$ as a (discrete-time) approximation of the gradient of a scalar potential $u : \R^{2n} \to \R$ evaluated at $\xi_j$, the learning task can be recast as a Hermite–Birkhoff (HB) interpolation problem:
\begin{align}\label{eqn:training}
    \nabla s(\xi_j) = y_j,
    \qquad j = 1,\dots,M.
\end{align}
Since the symplectic Euler rule only involves the gradient of the function $\Ham$, we only aim at good gradient approximation, while the absolute values of $s$ are of secondary importance, i.e., we do not aim to approximate the target Hamiltonian itself.
For background on HB interpolation for kernel-based models, we refer to \cite[Chapter~16.2]{wendland_2004}.

\noindent
Similar ideas have recently been studied, mainly in the context of neural networks.
Early work on structure-aware learning for Hamiltonian dynamics focused on identifying the
energy function itself. Hamiltonian Neural Networks (HNNs) \cite{greydanus2019hamiltonian}
learn a scalar Hamiltonian $\Ham_\mathrm{NN}$ from data and recover the dynamics via
\begin{equation}\label{eq:HNNs}
\dot x = J_{2n} \nabla \Ham_\mathrm{NN}(x),    
\end{equation}
thereby encoding conservation of $\Ham_\mathrm{NN}$ along the trajectories.
In \cite{DAVID2023112495}, the training procedure of HNNs is analyzed and a symplectic
training scheme is proposed that enforces discrete symplecticity via a
symplectic–integrator–based loss, yielding improved long-horizon stability and accuracy
compared to standard HNN training. A related approach for approximating time-series data by a learned Hamiltonian system is
presented in \cite{bertalan2019learning}, where a Gaussian process is employed for learning the real Hamiltonian and the dynamics are recovered similarly to \eqref{eq:HNNs}. Building on the idea of learning a Hamiltonian from data and coupling it
with symplectic time integration, \cite{Chen2020Symplectic} restricts to separable
Hamiltonians, $\Ham(q,p) = T(p) + V(q)$, and performs time stepping with a symplectic
integrator. This couples a learned, physics-informed model with a geometry-preserving
discretization, improving the long-time behavior. In \cite{zhong2019symplectic}, symplectic ODE-Nets are introduced, which enforce
Hamiltonian dynamics within the network architecture to learn the underlying dynamics.
By explicitly encoding the structure, these models achieve improved generalization with
fewer training samples.\\

\noindent
A number of approaches learn the flow directly as a symplectic map, without explicitly
recovering an underlying Hamiltonian. SympNets \cite{jin2020sympnets} compose simple
symplectic building blocks (rendering the resulting map symplectic by construction) to
approximate the time stepping map. In \cite{burby2020fast}, HénonNets are developed,
which concatenate Hénon-like maps and thus offer greater architectural flexibility than
SympNets. Furthermore, \cite{chen2021data} introduces Generating Function Neural Networks
(GFNNs), which learn a generating function whose associated canonical transformation
implicitly defines a symplectic map. More recently, \cite{Horn2024} unifies and generalizes many of these approaches via
Generalized Hamiltonian Neural Networks (GHNNs), which encompass separable-HNN models
\cite{Chen2020Symplectic}, direct symplectic-map learners
\cite{jin2020sympnets,burby2020fast}, and related architectures within a common framework
of enhanced expressivity.
\noindent
Several of these recent approaches address long–time step prediction by concatenating multiple learned,
separable Hamiltonians and composing their symplectic flows to
span a large step \cite{jin2020sympnets,Horn2024}. This leverages the fact that a concatenation of symplectic maps remains symplectic. Separable Hamiltonians are popular in these NN architectures because they align with many mechanical systems, yield an explicit and symplectic update (e.g., via symplectic Euler), and thus make training both computationally efficient (since no differentiation through an inner nonlinear solver is required) and structurally well-posed.

\noindent In contrast, we learn a single, general (non-separable) Hamiltonian
$s(q,p)$ which does not need to approximate the original Hamiltonian and perform prediction over a large horizon $\Delta T$ in one implicit symplectic–Euler–type step.

\noindent
Our key contributions are as follows:
\begin{enumerate}
    \item We propose a kernel-based scheme for learning and predicting Hamiltonian dynamics that is symplectic by construction and tailored to large time horizons.
    \item We provide a detailed analysis of existence and feasibility conditions, i.e., conditions under which a function satisfying the interpolation constraints exists and is uniquely determined. Moreover, we derive convergence results for the symplectic predictor, including a convergence analysis for first-derivative Hermite–Birkhoff interpolation.
    \item We combine the kernel model with model order reduction (MOR) \cite{benner2017model,benner2020model}, rendering the approach computationally feasible also for high-dimensional problems.
    \item We present numerical experiments that demonstrate the accuracy, long-time structure preservation, and efficiency of our method.
\end{enumerate}
\noindent For our method, one might ask whether an explicit predictor could be obtained by reconstructing a single, separable learned Hamiltonian of the form
\[
s(q,p) = s_p(p) + s_q(q).
\]
In this case, the corresponding Hermite–Birkhoff interpolation problem would formally read
\begin{align*}
    \frac{q_{\Delta T}^j - q_{0}^j}{\Delta T} &= \nabla_p s_{p}\left(p_{\Delta T}^j\right), \\
    -\frac{p_{\Delta T}^j - p_{0}^j}{\Delta T} &= \nabla_q s_{q}\left(q_{0}^j\right),
\end{align*}
with $(q_{\Delta T}^j, p_{\Delta T}^j)^\top = \Phi^{\Delta T}\left(q_0^j, p_0^j\right)^\top$. 
In general, this system is not well-posed: the right-hand side of the second equation can only learn a dependence on $q_0^j$, whereas the left-hand side typically depends on both $q_0^j$ and $p_0^j$ , i.e., the dependence on $p_0^j$ in the second equation cannot be reflected by this approach (and analogously for the first equation). 
\noindent In \Cref{fig:Data} we illustrate this issue by plotting the data points
\[
\left(p_{\Delta T}^j,  \frac{q_{\Delta T}^j - q_{0}^j}{\Delta T}\right)
\quad\text{and}\quad
\left(q_0^j, -\frac{p_{\Delta T}^j - p_{0}^j}{\Delta T}\right),
\]
and observe that these data are not well-suited for an interpolation approach based on a separable Hamiltonian since for each input value there are several output values. 
For comparison, we also present the same type of plot for the data points 
\[
\left((q_0^j, p_{\Delta T}^j),  \frac{q_{\Delta T}^j - q_{0}^j}{\Delta T}\right)
\quad\text{and}\quad
\left((q_0^j,p_{\Delta T}^j), -\frac{p_{\Delta T}^j - p_{0}^j}{\Delta T}\right),
\]
and observe that, in this case, there is only one output for each input.

\begin{figure}[h]
    \centering
     \includegraphics[width=0.6\linewidth]{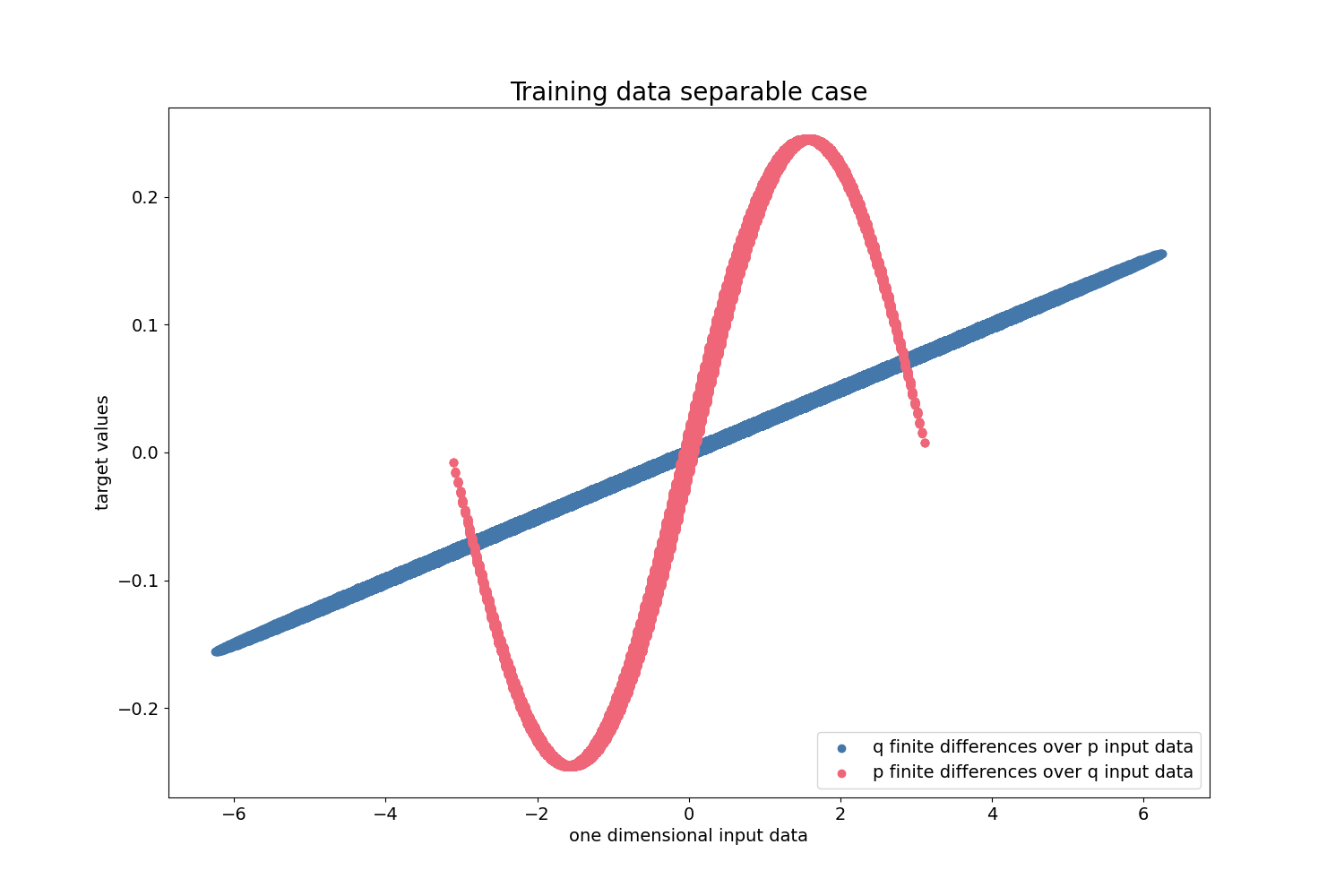}  \includegraphics[width=0.39\linewidth]{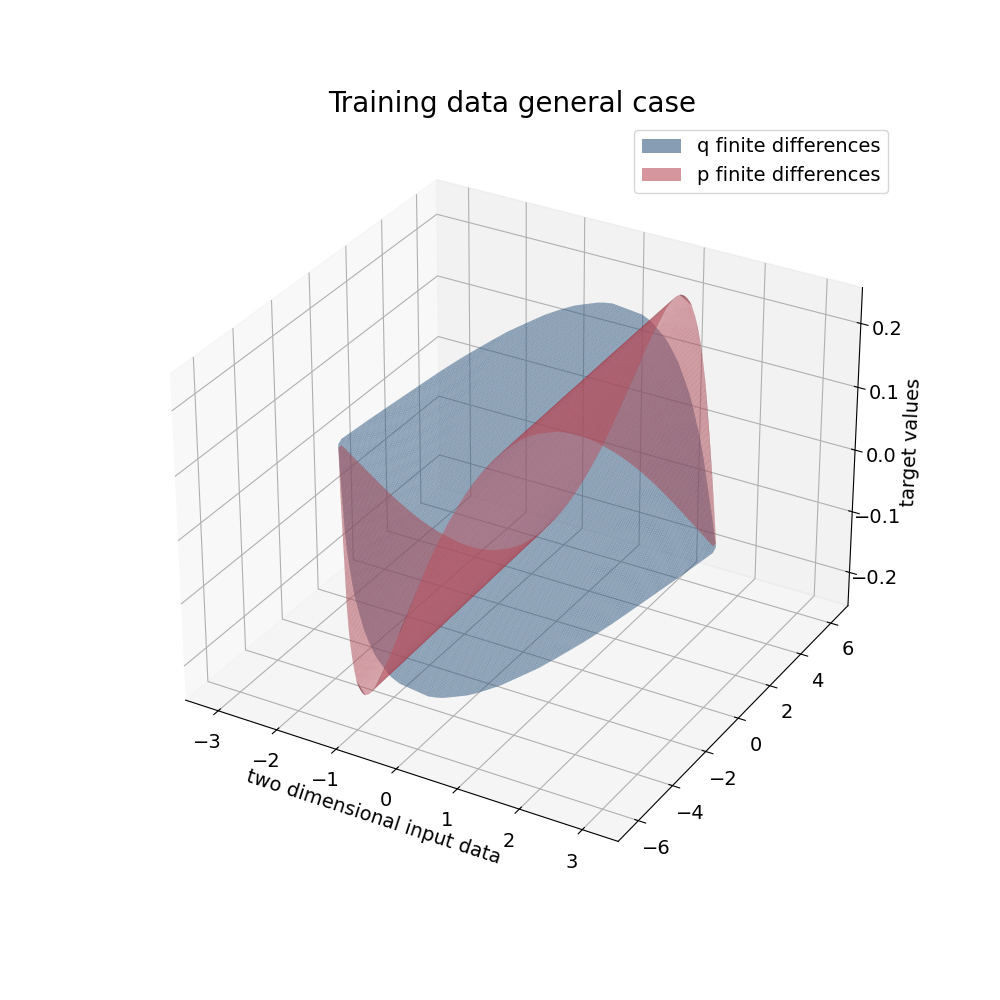}   
    \caption{Comparison of the training data for the explicit and implicit method with training data from the whole domain}
    \label{fig:Data}
\end{figure}

\noindent
The remainder of this work is structured as follows.
Section~\ref{Sec:Background} reviews kernel-based (generalized) interpolation.
Section~\ref{Sec:Analysis} investigates conditions under which an exact solution exists, which we subsequently approximate by our kernel-based scheme.
Section~\ref{Sec:Numerics} reports numerical results, and Section~\ref{Sec:Conclusion} concludes and outlines directions for future work.

\section{Background on Generalized Kernel Interpolation}\label{Sec:Background}

Let $\Omega$ be a nonempty set. A kernel is a symmetric function
$k : \Omega \times \Omega \to \R$. For a finite set of points
$X_M := \{x_1,\dots,x_M\} \subset \Omega$, the matrix
\[
K_{X_M} := \left(k(x_i,x_j)\right)_{i,j=1}^M \in \R^{M\times M}
\]
is called the Gramian matrix of $k$ (with respect to $X_M$).

\noindent We call $k$ positive definite (p.d.) if, for all $M \in \N$, all finite
sets $X_M \subset \Omega$, the Gramian matrix $K_{X_M}$ is symmetric positive semidefinite. 
We call $k$ {strictly positive definite (s.p.d.) if, for all $M \in \N$ and all sets
$X_M \subset \Omega$ consisting of pairwise distinct points, the matrix $K_{X_M}$ is symmetric positive definite.

\noindent
A reproducing kernel Hilbert space (RKHS) $H_k(\Omega)$ over $\Omega$ is a Hilbert space of
functions $f : \Omega \to \R$ in which all point evaluation functionals are continuous. 
For every RKHS there exists a function $k$ — the reproducing kernel — such that $k(\cdot,x) \in H_k(\Omega)$ and
the reproducing property
\[
f(x) = \langle f, k(\cdot,x)\rangle_{H_k(\Omega)}
\quad\text{for all } f \in H_k(\Omega),\ x \in \Omega,
\]
holds.

\noindent Moreover, the reproducing property extends to derivatives. In particular, if
$\Omega \subset \R^{2n}$ and $k \in C^{2}(\Omega \times \Omega)$, then
(for suitable $H_k(\Omega)$) the first-order partial derivative point evaluation functionals are continuous, and
\begin{equation}\label{eqn:rep_deriv}
\partial_\ell f(x)
  = \big\langle f, \partial_\ell^{(2)} k(\cdot,x)\big\rangle_{H_k(\Omega)}
  \quad\text{for all } f \in H_k(\Omega),\ x \in \Omega,\ \ell = 1,\dots,2n,
\end{equation}
where $\partial_\ell^{(2)}$ denotes the $\ell$-th partial derivative of $k$ with respect
to its second argument.

\noindent
Conversely, every p.d.\ kernel $k$ induces a unique RKHS $H_k(\Omega)$, i.e., there is a one-to-one correspondence between p.d.\ kernels and RKHSs.

\noindent In an RKHS $H_k(\Omega)$, we can formulate an abstract generalized interpolation problem of the form
\begin{align}\label{eq:min_norm_problem}
    s_M 
    = \argmin_{s \in H_k(\Omega)}
      \left\{ \| s \|_{H_k(\Omega)}  \big|  \lambda_j(s) = y_j \text{ for } j = 1,\dots,M \right\},
\end{align}
where $\lambda_1,\dots,\lambda_M \in H_k(\Omega)'$ are linearly independent, continuous linear functionals, and
$y_1,\dots,y_M$ are prescribed target values.

\noindent By the Riesz representation theorem in an RKHS, for each $\lambda_i$ there exists a unique
$v_i \in H_k(\Omega)$ such that
\[
\lambda_i(f) = \langle f, v_i\rangle_{H_k(\Omega)}
\quad\text{for all } f \in H_k(\Omega).
\]
If $k(\cdot,\cdot)$ is the reproducing kernel of $H_k(\Omega)$, then $v_i$ can be expressed explicitly as
\[
v_i(x)
  = \left(\lambda_i^{(2)} k\right)(x)
  = \lambda_i\left(k(x,\cdot)\right),
\]
i.e., $\lambda_i$ acts on the second argument of $k$. For a vector of functionals
$\Lambda_M = [\lambda_1,\dots,\lambda_M]$, we define the generalized Gramian matrix
$G_\Lambda \in \R^{M\times M}$ by
\[
(G_\Lambda)_{ij}
  := \langle v_j, v_i\rangle_{H_k(\Omega)}
  = \lambda_i^{(1)} \lambda_j^{(2)} k
  \quad (i,j = 1,\dots,M).
\]

\noindent By \cite[Theorem~16.1]{wendland_2004}, if $\lambda_1,\dots,\lambda_M$ are linearly independent on $H_k(\Omega)$, then for every $y \in \R^M$ there exists a unique minimum-norm interpolant (in the sense of \eqref{eq:min_norm_problem}), which has the form
\[
s(\cdot)
  = \sum_{j=1}^M c_j  v_j(\cdot)
  = \sum_{j=1}^M c_j  \lambda_j^{(2)} k(\cdot,\cdot),
\qquad\text{with}\quad
G_\Lambda c = y.
\]

\noindent A special case of the generalized interpolation problem is the Hermite–Birkhoff (HB) interpolation corresponding to \eqref{eqn:training}. 
For this, let $X_M = \{x_j\}_{j=1}^M \subset \Omega \subset \R^{2n}$ and, for each $j$, let
$\alpha_j \in \{1,\dots,2n\}$ be a coordinate index. We define linear functionals
\[
\lambda_{j}(f) := \lambda_{j,\alpha_j}(f) := \partial_{\alpha_j} f(x_j),
\qquad j = 1,\dots,M.
\]
These functionals are linearly independent on $H_k$ if, for $i \neq j$, either $\alpha_i \neq \alpha_j$ or $x_i \neq x_j$ and
$k$ is a strictly positive definite, translation-invariant kernel of the form
$k(x,x')=\phi(x-x')$ with $\phi\in L^1(\R^d)\cap C^{2}(\R^d)$; see
\cite[Theorem~16.4]{wendland_2004}.

\noindent Moreover, for $k \in C^{2}(\Omega \times \Omega)$, the functionals are continuous on $H_k(\Omega)$ by the derivative reproducing property \eqref{eqn:rep_deriv}.

\noindent In the HB setting, the Riesz representers and the corresponding generalized Gramian matrix entries are given by
\[
v_{j}(x) = \partial^{(2)}_{\alpha_j} k(x,x_j),\qquad
G_{i,j} = \partial^{(1)}_{\alpha_i} \partial^{(2)}_{\alpha_j} k(x_i,x_j).
\]
Hence, the interpolant takes the form
\begin{equation}\label{eqn:surr}
s(x) = \sum_{j=1}^M c_{j}  \partial^{(2)}_{\alpha_j} k(x,x_j),
\end{equation}
where the coefficients $c_{j}$, $j = 1,\dots,M$, are obtained from the linear system
\[
G_\Lambda c = y,
\]
with $c := (c_{1},\dots,c_{M})^\top$, data vector $y := (y_{1},\dots,y_{M})^\top$, and $G_\Lambda$ the generalized Gramian matrix.
If the family $\{\lambda_{j}\}_{j=1}^M$ is linearly independent, this system is uniquely solvable, and the interpolant exists uniquely \cite[Theorem~16.5]{wendland_2004}.

\subsection{Convergence rates  for gradient-Hermite-Birkhoff $f$-greedy}
\noindent
For large $M$, solving the dense linear system and evaluating
\eqref{eqn:surr} become computationally expensive and can suffer from
ill-conditioning. A common strategy is therefore to select a much smaller subset
$\Lambda_{m_{\max}} \subset \Lambda_M$ with $m_{\max} \ll M$.
Greedy algorithms—such as VKOGA~\cite{Santin2021VKOGA,wirtz2013vectorial}—
expand the active set $\Lambda_m$ iteratively according to a selection rule using the current interpolant $s_{f,\Lambda_m}$. A popular choice is the $f$-greedy rule, which starts with $\Lambda_0 = \emptyset$ and $s_{f,\Lambda_0} = 0$ and at iteration $m\geq 0$ selects 
\begin{equation}\label{eq:f-greedy}
  \lambda_{m+1}
  \in \operatorname*{argmax}_{\lambda_j \in \Lambda_M}
  \left|y_j - \lambda_j(s_{f,\Lambda_m})\right|,
  \qquad
  \Lambda_{m+1} := \Lambda_m \cup \{\lambda_{m+1}\},
\end{equation}
where the residuals at already selected functionals vanish, i.e., $y_j - \lambda_j(s_{f,\Lambda_m}) = 0, j = 1,...,m$.
Other sparsification strategies tailored to function interpolation (and not HB interpolation) include $\ell_1$-penalization~\cite{gao2010sparse}
and support vector regression (SVR)~\cite{scholkopf2002learning}.

\noindent In the context of Hermite–Birkhoff interpolation, we now formalize the
$f$-greedy strategy more precisely.

\noindent
For a finite selection 
\[
\{(x_i,\ell_i)\}_{i=1}^m \subset \Omega \times \mathcal J,
\]
with index set $\mathcal J := \{1,\dots,n\}$ (where $n$ is the ambient dimension of $\Omega$), define
\[
V_m := \Span \left\{\partial^{(2)}_{\ell_i} k(\cdot,x_i) : i = 1,\dots,m\right\},
\qquad
\Pi_m : H_k(\Omega) \to V_m
\]
to be the orthogonal projector onto $V_m$.
Given a target function $u \in H_k(\Omega)$, we set $s_m := \Pi_m u$ and $e_m := u - s_m$.

\noindent Then the $f$-greedy selection can be formulated as follows:
for $m \ge 0$, choose
\[
(x_{m+1},\ell_{m+1})
  \in \operatorname*{argmax}_{x \in \Omega,\ \ell \in \mathcal J}
  \left|\partial_\ell e_m(x)\right|,
\]
and set $s_{m+1} := \Pi_{m+1} u$.

\noindent
Greedy schemes produce sparse surrogates and come with rigorous—sometimes even optimal—
convergence guarantees~\cite{wenzel2023analysis,santin2018greedy,santin2024optimality}.
For HB interpolation, greedy convergence has been investigated in
\cite{ehring2024hermite} and, in more general functional settings, in
\cite{schaback2019greedy,albrecht2025convergence}. The application-oriented work
\cite{ehring2024hermite} establishes convergence of a target-dependent greedy Hermite
kernel scheme (based on an $f$-type selection rule) for value-function surrogates in
optimal control. In contrast, \cite{schaback2019greedy} studies well-posed operator
equations in an abstract Hilbert-space framework and relates greedy convergence rates
to Kolmogorov $n$-widths of the underlying representer set, without specializing to
the Hermite case. Complementarily, \cite{albrecht2025convergence} develops a
convergence theory for generalized kernel-based interpolation with totally bounded
sets of sampling functionals, which in particular covers HB
interpolation.

\noindent In the following, we
extend the $f$-greedy convergence rates from \cite{wenzel2023analysis,wenzel2025adaptive}
to our gradient–HB case. We begin with the following theorem, which provides a bound on the maximum derivative error.

\begin{theorem}[First-order HB bound]\label{thm:HB-mean}
Let $u \in H_k(\Omega)$, $\Omega \subset \R^{n}$, and apply the $f$-greedy algorithm of the previous subsection with index set $\mathcal J := \{1,\dots,2n\}$. Then, for every $m \ge 1$,
\begin{align}\label{thm:hb:bound}
\min_{m+1\le i\le 2m}\ \|\nabla e_i\|_{L^\infty(\Omega)}
\ \le\ \sqrt{n}  m^{-1/2} \|e_{m+1}\|_{H_k(\Omega)}
\left[\ \prod_{i=m+1}^{2m} P_i(x_{i+1},\ell_{i+1})\ \right]^{1/m},
\end{align}
where the (derivative) power function is defined by
\[
P_m(x,\ell) := \big\|(I-\Pi_m) \partial_\ell^{(2)} k(\cdot,x)\big\|_{H_k(\Omega)},
\qquad (x,\ell)\in \Omega\times\mathcal J.
\]
\end{theorem}

\begin{proof}

The proof proceeds similarly to the $\beta$-greedy analysis in
\cite{wenzel2023analysis,wenzel2025adaptive} (with $\beta = 1$).
\end{proof}
\noindent For completeness it is given in the appendix under
\hyperlink{proof:HB-mean}{Proof of Theorem~\ref{thm:HB-mean}}. 

\noindent In the remainder of this section, we study the convergence of the symplectic predictor. 
To this end, we introduce a type~II generating function, i.e., a generating function expressed in the mixed variables $(q_0,p_{\Delta T})$.
For background on the different generating-function types, see \cite{cline2017variational}.

\begin{definition}[Type~II generating function]\label{def:typeII}
Let $\Phi^{\Delta T}:D\to\R^{2n}$ and define
\[
\Omega := \left\{ I_1 x_0 + I_2 \Phi^{\Delta T}(x_0) : x_0 \in D \right\},
\qquad
I_1:=\begin{bmatrix} I_n & 0 \\[2pt] 0 & 0 \end{bmatrix},\quad
I_2:=\begin{bmatrix} 0 & 0 \\[2pt] 0 & I_n \end{bmatrix}.
\]
A function $S^{\Delta T}\in H_k(\Omega)\cap C^1(\Omega)$ is called a \emph{type~II generating function}
(for $\Phi^{\Delta T}$ on $\Omega$) if, for all $x_0\in D$,
\begin{align}\label{eqn:mixed-training}
    J_{2n}^{\top} \frac{\Phi^{\Delta T}(x_0)-x_0}{\Delta T}
    = \nabla S^{\Delta T}\!\left(I_1 x_0 + I_2 \Phi^{\Delta T}(x_0)\right).
\end{align}
\end{definition}
\noindent With Definition~\ref{def:typeII} in place, we establish how the convergence rate for the gradient approximation in \Cref{thm:HB-mean} propagates to a corresponding convergence rate for the symplectic predictor.

\begin{theorem}[Convergence rate for the prediction error]\label{thm:residual-to-prediction}
Let $S^{\Delta T}\in C^{2}(\Omega)$ be a type-II generating function of the time--$\Delta T$
Hamiltonian flow $\Phi^{\Delta T}$, defined on an open, convex set $\Omega\subset\mathbb{R}^{2n}$ of
mixed variables $(q_0,p_{\Delta T})$. Assume $S^{\Delta T}\in H_k(\Omega)$ and that
$s_{i_m}\in H_k(\Omega)$ is obtained by the HB $f$-greedy procedure from the previous subsection,
where we choose an index \( i_m \in \{m+1,\dots,2m\} \) such that
\[
\|\nabla e_{i_m}\|_{L^\infty(\Omega)}
=
\min_{m+1\le i\le 2m}\ \|\nabla e_i\|_{L^\infty(\Omega)}.
\]

\noindent Fix a compact set $K\subset\mathbb{R}^{2n}$ and assume:

\begin{enumerate}
\item[(i)] \textbf{Uniform solvability.}
For every $x_0\in K$ and every $m$ sufficiently large, the implicit equations
\[
y = x_0 + \Delta T J_{2n} \nabla S^{\Delta T}(I_1x_0+I_2y),
\qquad
y = x_0 + \Delta T J_{2n} \nabla s_{i_m}(I_1x_0+I_2y),
\]
admit unique solutions $y^\ast=\Phi^{\Delta T}(x_0)$ and
$y_m=x_{\Delta T,i_m}(x_0)$, respectively, and satisfy
\[
I_1x_0+I_2y^\ast\in \Omega,\qquad I_1x_0+I_2y_m\in \Omega.
\]

\item[(ii)] \textbf{Uniform Lipschitz continuity in $y$.}
There exists $L_S>0$ such that
\[
\sup_{\substack{x_0\in K\\ y:\ I_1x_0+I_2y\in \Omega}}
\bigl\|\nabla^2 S^{\Delta T}(I_1 x_0 + I_2 y)\, I_2\bigr\|_2
= L_S < \infty .
\]
\end{enumerate}
Assume that
\begin{equation}\label{eqn:deltaT}
    \Delta T < \frac{1}{L_S}.
\end{equation}

\noindent Then there exists a constant $C>0$, independent of $m$ and $x_0$, such that
\[
\sup_{x_0\in K}\ \|x_{\Delta T,i_m}(x_0)-\Phi^{\Delta T}(x_0)\|_2
\ \le\ C\,\Delta T\,\sqrt{2n}\, m^{-1/2}\, \|e_{m+1}\|_{H_k(\Omega)}
\left[\ \prod_{i=m+1}^{2m} P_i(x_{i+1},\ell_{i+1})\ \right]^{1/m}.
\]
\end{theorem}

\begin{proof}
Fix $x_0\in K$ and denote by $y^\ast=\Phi^{\Delta T}(x_0)$ and
$y_m=x_{\Delta T,i_m}(x_0)$ the exact and approximate predictors, respectively.
Define
\[
F_m(y;x_0)
:= y-x_0-\Delta T J_{2n} \nabla s_{i_m}(I_1x_0+I_2y),
\qquad
F(y;x_0)
:= y-x_0-\Delta T J_{2n} \nabla S^{\Delta T}(I_1x_0+I_2y).
\]

\noindent By construction,
\[
0 = F_m(y_m;x_0)
= y_m-x_0-\Delta T J_{2n} \nabla s_{i_m}(I_1x_0+I_2y_m).
\]
Subtracting from $F(y_m;x_0)$ yields
\[
F(y_m;x_0)
=
-\Delta T J_{2n}
\bigl(\nabla S^{\Delta T}-\nabla s_{i_m}\bigr)(I_1x_0+I_2y_m)
=
-\Delta T J_{2n}\nabla e_{i_m}(I_1x_0+I_2y_m).
\]
Hence, using $\|J_{2n}\|_2=1$,
\begin{equation}\label{eqn:boundFym}
\|F(y_m;x_0)\|_2
\le
\Delta T\,\|\nabla e_{i_m}\|_{L^\infty(\Omega)}.
\end{equation}

\noindent Since $F(y^\ast;x_0)=0$, the integral form of the multivariate mean value theorem gives
\[
F(y_m;x_0)
=
\int_0^1
D_y F\bigl(y^\ast+\theta (y_m-y^\ast);x_0\bigr)\,d\theta\,(y_m-y^\ast).
\]
\noindent with
\[
D_yF(y;x_0)
:= I_{2n} - \Delta T J_{2n} \nabla^2 S^{\Delta T}(I_1 x_0 + I_2 y) I_2.
\]
Define
\[
A_m(x_0)
:=
\int_0^1
D_y F\bigl(y^\ast+\theta (y_m-y^\ast);x_0\bigr)\,d\theta.
\]

\noindent 
Set
\[
B(y;x_0)
:= \Delta T\,J_{2n}\,\nabla^2 S^{\Delta T}(I_1x_0+I_2y)\,I_2,
\textrm{ such that }
D_yF(y;x_0)=I_{2n}-B(y;x_0).
\]
Define the averaged operator
\[
\bar B_m(x_0)
:=
\int_0^1
B\bigl(y^\ast+\theta (y_m-y^\ast);x_0\bigr)\,d\theta,
\textrm{ such that }
A_m(x_0)=I_{2n}-\bar B_m(x_0).
\]

\noindent Since $\Omega$ is convex and
$I_1x_0+I_2y^\ast,\ I_1x_0+I_2y_m\in\Omega$, the entire segment lies in $\Omega$.
By assumption~(ii),
\[
\|\bar B_m(x_0)\|_2
\le
\Delta T\,L_S < 1.
\]
Therefore $A_m(x_0)$ is invertible and admits the Neumann series
\[
A_m(x_0)^{-1}
=
\sum_{k=0}^\infty \bar B_m(x_0)^k,
\]
\noindent and further
\[
\|A_m(x_0)^{-1}\|_2
\le
\frac{1}{1-\Delta T L_S}
=: C_{\mathrm{inv}}.
\]

\noindent Consequently,
\[
y_m-y^\ast
=
A_m(x_0)^{-1}F(y_m;x_0),
\]
\noindent
and
\[
\|y_m-y^\ast\|_2
\le
C_{\mathrm{inv}}\,\|F(y_m;x_0)\|_2.
\]
Combining with the bound \eqref{eqn:boundFym} and Theorem~\ref{thm:HB-mean} yields
\[
\sup_{x_0\in K}\ \|x_{\Delta T,i_m}(x_0)-\Phi^{\Delta T}(x_0)\|_2
\ \le\ C\,\Delta T\,\sqrt{2n}\, m^{-1/2}\, \|e_{m+1}\|_{H_k(\Omega)}
\left[\ \prod_{i=m+1}^{2m} P_i(x_{i+1},\ell_{i+1})\ \right]^{1/m},
\]
with $C:=C_{\mathrm{inv}}$.
\end{proof}

\noindent In summary, the gradient--HB $f$-greedy strategy yields a quantitative decay bound for the maximal derivative error (\Cref{thm:HB-mean}), and \Cref{thm:residual-to-prediction} shows that this decay translates to a uniform convergence rate for the symplectic predictor on compact subsets.

\section{On the Existence of the Target Function}\label{Sec:Analysis}
\noindent
To obtain provable convergence guarantees, we assume that the generating function belongs to the RKHS $H_k(\Omega)$.
In our setting, however, it is not a priori clear under which conditions there even exists a differentiable
function $S^{\Delta T}$ that satisfies the interpolation constraint
\begin{align}\label{eqn:mixed-training}
    J_{2n}^{\top} \frac{\Phi^{\Delta T}(x_0)-x_0}{\Delta T}
    = \nabla S^{\Delta T} \big(I_1 x_0 + I_2 \Phi^{\Delta T}(x_0)\big).
\end{align}

\noindent When such a function $S^{\Delta T}$ exists and, in addition, $S^{\Delta T} \in H_k(\Omega)$ on the mixed-data domain
\[
\Omega := \left\{ I_1 x_0 + I_2 \Phi^{\Delta T}(x_0) : x_0 \in D \subset \R^{2n} \right\},
\] 
the HB interpolation convergence analysis from the previous section applies.
Existence of such an $S^{\Delta T}$ is not automatic: the discrete flow increment must be representable as
the gradient of a scalar function of the mixed variables $(q_0,p_{\Delta T})$.
We make this precise and provide verifiable conditions in this section. 
First, we provide a general statement with abstract conditions on the flow map
$\Phi^{\Delta T}$ under which a generating function exists. Since the proof of
the following theorem involves locally inverting the flow map, the variables
$q,p$ become functions of the output variables. To indicate whether a symbol
is treated as a function or as a variable, we display functions in bold
throughout this section.

\begin{theorem}[Existence of a generating function]
\label{thm:genInt}
Let $\Ham \in C^2(\Omega)$, $\Omega \subset \R^{2n}$, and let
$\Phi^{\Delta T} : \Omega \to \Phi^{\Delta T}(\Omega)$
be the flow map for a fixed $\Delta T > 0$.
Write \footnote{Note that, for brevity, we suppress the explicit dependence on $\Delta T$ in
$\mathbf Q,\mathbf P$ and their derivatives
whenever no confusion can arise.}
\[
\Phi^{\Delta T}(q,p)
  = \left(\mathbf Q(q,p), \mathbf P(q,p)\right),
\]
and denote the Jacobian of $\Phi^{\Delta T}$ with respect to $(q,p)$ by
\[
D\Phi^{\Delta T}(q,p)
=
\begin{pmatrix}
\frac{\partial}{\partial q} \mathbf Q  & \frac{\partial}{\partial p} \mathbf Q  \\[1pt]
\frac{\partial}{\partial q} \mathbf P  & \frac{\partial}{\partial p} \mathbf P  
\end{pmatrix}
=:
\begin{pmatrix}
\mathbf A(q,p) & \mathbf B(q,p)\\[2pt]
\mathbf C(q,p) & \mathbf D(q,p)
\end{pmatrix},
\]
where $\mathbf A,\mathbf B,\mathbf C,\mathbf D : \Omega \to \R^{n\times n}$.

\noindent Fix an open set $U \subset \Omega$ such that:
\begin{itemize}
\item[(i)] $\mathbf D(q,p)$ is invertible for all $(q,p) \in U$;
\item[(ii)] The map
\[
\boldsymbol\Psi : U \to W,\qquad
\boldsymbol\Psi(q,p) := \left(q,\mathbf P(q,p)\right),
\]
is a diffeomorphism onto a simply connected open set $W \subset \R^{2n}$.
\end{itemize}
Then there exist functions $\mathbf p, \mathbf Q \in C^1(W,\R^n)$, unique up to an additive
constant in $S$, such that, if we view them as functions of $(q,P) \in W$ via
\[
(q,P) = \boldsymbol\Psi(q,p),\qquad
(Q,P) = \Phi^{\Delta T}(q,p),
\]
i.e.\ $(q,p) = \boldsymbol\Psi^{-1}(q,P)$ and $(Q,P) = \Phi^{\Delta T}(q,\mathbf p(q,P))
= (\mathbf Q(q,P),P)$, then
\[
\mathbf p(q,P) = \partial_q S(q,P),\qquad
\mathbf Q(q,P) = \partial_P S(q,P).
\]

\noindent Moreover, 
the type–II generating function 
\[
S^{\Delta T}(q,P) := \frac{1}{\Delta T}\left(S(q,P) - q^\top P\right)
\]
satisfies, for all $x = (q,p) \in U$,
\[
J_{2n}^\top \frac{\Phi^{\Delta T}(x) - x}{\Delta T}
= \nabla_{(q,P)} S^{\Delta T}\left(I_1 x + I_2 \Phi^{\Delta T}(x)\right).
\]

\begin{proof}
Recall $\Psi(q,p) := (q,\P(q,p))$ and
\[
D\Psi(q,p)
= \begin{pmatrix}
\frac{\partial}{\partial q}  q & \frac{\partial}{\partial p} q\\[1pt]
\frac{\partial } {\partial q} \P& \frac{\partial }{\partial p}\P
\end{pmatrix}
=
\begin{pmatrix}
I_n & 0 \\
\C(q,p) & \D(q,p)
\end{pmatrix}.
\]
By assumption (i), $\D(q,p)$ is invertible on $U$, so $D\Psi(q,p)$ is invertible and, by the
Implicit Function Theorem, $\Psi$ is locally invertible on $U$. Hence there exists a $C^1$
function $\p = \p(q,P)$ on $W$ such that
\[
\P\left(q,\p(q,P)\right) = P.
\]
We then define
\[
\Q(q,P) := \Q\left(q,\p(q,P)\right).
\]

\noindent
Since $\Phi^{\Delta T}$ is symplectic, we have
\[
(D\Phi^{\Delta T}(q,p))^\top J_{2n} D\Phi^{\Delta T}(q,p) = J_{2n}.
\]
Writing
\[
D\Phi^{\Delta T}(q,p)
=
\begin{pmatrix}
\A & \B\\[1pt]
\C & \D
\end{pmatrix},
\]
and using the symplecticity condition this implies
\begin{align*}
\begin{pmatrix}0 & I_n\\ -I_n & 0\end{pmatrix} = J_{2n} = (D\Phi^{\Delta T})^\top J_{2n} D\Phi^{\Delta T}
&=
\begin{pmatrix}\A^\top & \C^\top\\ \B^\top & \D^\top\end{pmatrix}
\begin{pmatrix}0 & I_n\\ -I_n & 0\end{pmatrix}
\begin{pmatrix}\A & \B\\ \C & \D\end{pmatrix}  \\
&=
\begin{pmatrix}
\A^\top \C - \C^\top \A & \A^\top \D - \C^\top \B\\[2pt]
\B^\top \C - \D^\top \A & \B^\top \D - \D^\top \B
\end{pmatrix},
\end{align*}
holding pointwise.

\noindent
With $\D$ invertible, these identities imply
\begin{equation}\label{eqn:mateqn2}
\B\D^{-1} = (\B\D^{-1})^\top,\qquad
\A - \B\D^{-1}\C = \D^{-\top},\qquad
\D^{-1}\C = (\D^{-1}\C)^\top.
\end{equation}
Indeed:
\begin{enumerate}
    \item  From $\B^\top \D = \D^\top \B$, multiplying on the right by $\D^{-1}$ and on the left by
  $\D^{-\top}$ gives
  \[
  \D^{-\top} \B^\top = \B\D^{-1},
  \]
  i.e., $\B\D^{-1}$ is symmetric.

\item From $\B^\top \C - \D^\top \A = -I_n$, multiplying on the left by $\D^{-\top}$ yields
  \[
  \D^{-\top}\B^\top \C - \A = -\D^{-\top}.
  \]
  Using $\D^{-\top}\B^\top = (\B\D^{-1})^\top = \B\D^{-1}$ (by symmetry of $\B\D^{-1}$), we obtain
  \[
  \A = \B\D^{-1}\C + \D^{-\top}
  \quad\Longrightarrow\quad
  \A - \B\D^{-1}\C = \D^{-\top}.
  \]

\item Finally, starting from $\A^\top\C = \C^\top\A$ and inserting
  $\A = \D^{-\top} + \B\D^{-1}\C$ gives
  \begin{align*}
  \A^\top \C
  &= (\D^{-\top} + \B\D^{-1}\C)^\top \C
   = \D^{-1}\C + \C^\top \D^{-\top} \B^\top \C,\\
  \C^\top \A
  &= \C^\top(\D^{-\top} + \B\D^{-1}\C)
   = \C^\top \D^{-\top} + \C^\top \B\D^{-1}\C.
  \end{align*}
  Since $\D^{-\top}\B^\top = \B\D^{-1}$, the second terms agree, and we conclude
  \[
  \D^{-1}\C = \C^\top \D^{-\top}
  \quad\Longleftrightarrow\quad
  (\D^{-1}\C)^\top = \D^{-1}\C.
  \]

\end{enumerate}
\noindent
By the local invertibility of $\Psi$, there is a local solution $\p = \p(q,P)$ of
$\P(q,p) = P$, i.e.,
\[
\P\left(q,\p(q,P)\right) = P.
\]
We recall the definition
\[
\Q(q,P) = \Q\left(q,\p(q,P)\right),
\]
and differentiate the identities
\[
\P\left(q,\p(q,P)\right) = P,\qquad
\Q\left(q,\p(q,P)\right) = Q
\]
with respect to $q$ and $P$.

\noindent From $\P(q,\p(q,P)) = P$ we obtain
\begin{align*}
\partial_q \P + \partial_p \P \partial_q \p &= 0 
\quad\Longrightarrow\quad \C + \D \frac{\partial \p}{\partial q} = 0
\quad\Longrightarrow\quad \frac{\partial \p}{\partial q} = -\D^{-1}\C,\\[2pt]
\partial_p \P \partial_P \p &= I_n 
\quad\Longrightarrow\quad \D \frac{\partial \p}{\partial P} = I_n
\quad\Longrightarrow\quad \frac{\partial \p}{\partial P} = \D^{-1}.
\end{align*}

\noindent From $\Q(q,\p(q,P)) = Q$ we get
\begin{align*}
\partial_q \Q + \partial_p \Q \partial_q \p
&= \A + \B\left(-\D^{-1}\C\right)
= \A - \B\D^{-1}\C
= \frac{\partial \Q}{\partial q},\\[2pt]
\partial_p \Q \partial_P \p
&= \B\D^{-1}
= \frac{\partial \Q}{\partial P}.
\end{align*}
Collecting these four blocks yields
\[
\frac{\partial(\p,\Q)}{\partial(q,P)} =
\begin{pmatrix}
-\D^{-1}\C & \D^{-1}\\[2pt]
\A-\B\D^{-1}\C & \B\D^{-1}
\end{pmatrix}.
\]

\noindent
Using the identities in \eqref{eqn:mateqn2}, this Jacobian satisfies
\[
\frac{\partial \p}{\partial q}
= \left(\frac{\partial \p}{\partial q}\right)^\top,\qquad
\frac{\partial \Q}{\partial P}
= \left(\frac{\partial \Q}{\partial P}\right)^\top,\qquad
\frac{\partial \p}{\partial P}
= \left(\frac{\partial \Q}{\partial q}\right)^\top
\quad\left(\text{since }(\A-\B\D^{-1}\C)^\top = \D^{-1}\right).
\]
These are precisely the symmetry conditions that are necessary for the existence of a scalar
potential $S$ with $\p = \partial_q S$ and $\Q = \partial_P S$.

\noindent
We now construct such a potential. Fix a reference point $(q_0,P_0) \in W$ and define
\[
S(q,P) := \mathcal I_1(q) + \mathcal I_2(q,P),
\]
with
\[
\mathcal I_1(q) := \int_0^1 \p\left(q_0+\tau(q-q_0),P_0\right)^\top (q-q_0) d\tau,
\quad
\mathcal I_2(q,P) := \int_0^1 \Q\left(q,P_0+\tau(P-P_0)\right)^\top (P-P_0) d\tau.
\]
Clearly $S(q_0,P_0) = 0$.

\noindent
For $\mathcal I_1$, define
\[
g(\tau) := \tau \p\left(q_0+\tau(q-q_0),P_0\right).
\]
Then
\[
g'(\tau)
= \p\left(q_0+\tau(q-q_0),P_0\right)
  + \tau (\partial_q \p)\left(q_0+\tau(q-q_0),P_0\right)^\top (q-q_0),
\]
and thus
\begin{align*}
\partial_q \mathcal I_1(q)
&= \int_0^1 \left[\p\left(q_0+\tau(q-q_0),P_0\right)
 + \tau (\partial_q \p)\left(q_0+\tau(q-q_0),P_0\right)^\top (q-q_0)\right] d\tau \\
&= \int_0^1 g'(\tau) d\tau
= g(1)-g(0)
= \p(q,P_0).
\end{align*}

\noindent
For $\mathcal I_2$, we note that
\begin{align*}
\partial_q\left(\Q\left(q,P_0+\tau(P-P_0)\right)^\top (P-P_0)\right)
&= \left[(\partial_q \Q)\left(q,P_0+\tau(P-P_0)\right)\right]^\top (P-P_0) \\
&= (\partial_P \p)\left(q,P_0+\tau(P-P_0)\right)(P-P_0),
\end{align*}
where we used the relation $(\partial_q \Q)^\top = \partial_P \p$.
Furthermore,
\[
\frac{d}{d\tau} \p\left(q,P_0+\tau(P-P_0)\right)
= (\partial_P \p)\left(q,P_0+\tau(P-P_0)\right)(P-P_0),
\]
so
\begin{align*}
\partial_q \mathcal I_2(q,P) 
&= \int_0^1 (\partial_P \p)\left(q,P_0+\tau(P-P_0)\right)(P-P_0) d\tau \\
&= \int_0^1 \frac{d}{d\tau} \p\left(q,P_0+\tau(P-P_0)\right) d\tau \\
&= \p(q,P)-\p(q,P_0).
\end{align*}

\noindent
Combining the two terms, we obtain
\[
\partial_q S(q,P)
= \partial_q \mathcal I_1(q) + \partial_q \mathcal I_2(q,P)
= \p(q,P_0)+\left(\p(q,P)-\p(q,P_0)\right)
= \p(q,P).
\]

\noindent
Since $\mathcal I_1$ does not depend on $P$, we have $\partial_P \mathcal I_1 = 0$.
For $\mathcal I_2$, define
\[
h(\tau) := \tau \Q\left(q,P_0+\tau(P-P_0)\right).
\]
Then
\[
h'(\tau)
= \Q\left(q,P_0+\tau(P-P_0)\right)
  + \tau (\partial_P \Q)\left(q,P_0+\tau(P-P_0)\right)(P-P_0),
\]
and hence
\begin{align*}
\partial_P \mathcal I_2(q,P)
&= \int_0^1 \left[(\partial_P \Q)\left(q,P_0+\tau(P-P_0)\right) \tau(P-P_0)
       + \Q\left(q,P_0+\tau(P-P_0)\right)\right] d\tau \\
&= \int_0^1 h'(\tau) d\tau
 = h(1)-h(0)
 = \Q(q,P).
\end{align*}
Thus
\[
\partial_P S(q,P)
= \partial_P \mathcal I_1(q) + \partial_P \mathcal I_2(q,P)
= \Q(q,P).
\]
Consequently,
\[
\partial_q S(q,P) = \p(q,P),\qquad
\partial_P S(q,P) = \Q(q,P).
\]

\noindent
Finally, define
\[
S^{\Delta T}(q,P) := \frac{1}{\Delta T}\left(S(q,P) - q^\top P\right),
\]
which yields
\[
\nabla_q S^{\Delta T}(q,P)
= \frac{\p(q,P) - P}{\Delta T},\qquad
\nabla_P S^{\Delta T}(q,P)
= \frac{\Q(q,P) - q}{\Delta T}.
\]
Evaluating at
\[
(q,P) = I_1 x + I_2 \Phi^{\Delta T}(x), \quad x = (q,p),
\]
we obtain
\[
\nabla_{(q,P)} S^{\Delta T}\left(I_1 x + I_2 \Phi^{\Delta T}(x)\right)
=
\begin{pmatrix}
-(P-p)/\Delta T\\[2pt]
(Q-q)/\Delta T
\end{pmatrix}
= J_{2n}^\top \frac{\Phi^{\Delta T}(x) - x}{\Delta T},
\]
which is the desired identity.
\end{proof}

\end{theorem}
\noindent Building on the preceding result, we show that the time--$\Delta T$ flow map of a Hamiltonian system admits a generating function on any forward-invariant set, provided that $\Delta T$ satisfies an explicit step-size restriction. In particular, for $C^2$ Hamiltonians, choosing $\Delta T$ sufficiently small guarantees the existence of a generating function on the considered domain.

\begin{theorem}[Uniform invertibility on compact forward-invariant sets]\label{thm:forward-compact}
Let $\Ham \in C^2(\R^{2n})$ and consider the Hamiltonian system
\[
\dot x = J_{2n}\nabla \Ham(x),
\]
with flow $\Phi^t$. Let $K \subset \R^{2n}$ be nonempty, compact, and forward invariant
for times in an interval $[0,T]$ (with $T \in (0,\infty]$), i.e.\ $\Phi^t(K) \subset K$ for all
$t \in [0,T]$. Define
\[
L_K := \sup_{y \in K}\big\| \nabla^2 \Ham(y) \big\|_2 < \infty.
\]
Then, for every $x \in K$ and every fixed $\Delta T$ satisfying \footnote{
interpret $\tfrac{\log 2}{L_K} = +\infty \text{ if } L_K = 0$}
\[
0 \le \Delta T < \Delta T_K^\star := \min \left\{ T,\ \frac{\log 2}{L_K}\right\},
\]
on any $(q,P)$-chart $(U,\psi)$ with $U\cap K\neq\emptyset$ such that the coordinate image $\psi(U)\subset\mathbb{R}^{2n}$ is simply connected,
the assumptions of Theorem~\ref{thm:genInt} are satisfied, and there exists a generating
function $S(q,P)$ such that, viewing $p$ and $Q$ as functions of $(q,P)$,
\[
\mathbf p(q,P) = \partial_q S(q,P),\qquad
\mathbf Q(q,P) = \partial_P S(q,P),
\]
and the generating-function identity from Theorem~\ref{thm:genInt} holds uniformly for all
$x \in K$ and all $0 \le \Delta T < \Delta T_K^\star$.
\end{theorem}
\begin{proof}
Fix $x \in K$ and $\Delta T \in [0,T)$. By forward invariance, the trajectory segment
$\{\Phi^t(x) : t \in [0,\Delta T]\}$ lies in $K$. Let
\[
\mathbf Y(t;x) := D\Phi^t(x).
\]
Then $\mathbf Y$ solves the differential equation
\[
\dot{\mathbf Y}(t;x) = D F(\Phi^t(x)) \mathbf Y(t;x),\qquad \mathbf Y(0;x) = I_{2n},
\]
with $F(z) = J_{2n} \nabla \Ham(z)$ and $D F(z) = J_{2n} \nabla^2 \Ham(z)$.
Since $\nabla^2 \Ham$ is continuous and the trajectory remains in $K$, we have
\[
\|D F(\Phi^t(x))\|_2
= \left\|J_{2n}\nabla^2 \Ham(\Phi^t(x))\right\|_2
\le \|\nabla^2 \Ham(\Phi^t(x))\|_2
\le L_K
\]
for all $t \in [0,\Delta T]$ (using $\|J_{2n}\|_2 = 1$). Grönwall’s inequality then yields
\[
\|\mathbf Y(\Delta T;x) - I_{2n}\|_2 \le e^{L_K \Delta T} - 1.
\]
If $\Delta T < \log(2)/L_K$ (or for all $\Delta T$ if $L_K = 0$), then $e^{L_K\Delta T} - 1 < 1$.

\noindent Write
\[
\mathbf Y(\Delta T;x)
=
\begin{pmatrix}
\mathbf A & \mathbf B\\
\mathbf C & \mathbf D
\end{pmatrix},
\]
and define $E := \begin{pmatrix}0\\ I_n\end{pmatrix} \in \R^{2n \times n}$. Then
\[
\mathbf D - I_n
= E^\top\left(\mathbf Y(\Delta T;x) - I_{2n}\right)E,
\]
so
\[
\|\mathbf D - I_n\|_2
= \left\|E^\top\left(\mathbf Y(\Delta T;x) - I_{2n}\right)E\right\|_2
\le \|E^\top\|_2 \|\mathbf Y(\Delta T;x) - I_{2n}\|_2 \|E\|_2
\le \|\mathbf Y(\Delta T;x) - I_{2n}\|_2 < 1,
\]
since $\|E\|_2 = \|E^\top\|_2 = 1$.
Hence $\mathbf D = I_n + (\mathbf D - I_n)$ is invertible by the Neumann series.

\noindent Invertibility of the lower-right block $\mathbf D(\Delta T;x)$ implies that, near $x$, the map
\[
(q,p) \mapsto \left(q,\mathbf P(q,p;\Delta T)\right)
\]
is a local diffeomorphism (its Jacobian is block upper-triangular with diagonal blocks $I_n$ and $\mathbf D$).
Thus, on any simply connected $(q,P)$-chart intersecting $K$, the assumptions of
Theorem~\ref{thm:genInt} are satisfied. Consequently, there exists a generating function
$S(q,P)$ such that, viewing $p$ and $Q$ as functions of $(q,P)$,
\[
\mathbf p(q,P) = \partial_q S(q,P),\qquad
\mathbf Q(q,P) = \partial_P S(q,P),
\]
and the generating-function identity from Theorem~\ref{thm:genInt} holds for all $x \in K$ and
all $0 \le \Delta T < \Delta T_K^\star$. The bounds are uniform in $x \in K$ because $L_K$
was defined as a supremum over the compact set $K$.
\end{proof}

\begin{remark}\label{rem:fwd-invariance}
For a Hamiltonian system $\dot x = J_{2n}\nabla \Ham(x)$ the Hamiltonian is conserved, i.e.,
$\Ham(\Phi^t(x_0)) = \Ham(x_0)$ for all $t$ in the interval of existence. Thus any sublevel set
\[
K_E := \{ x \in \R^{2n} : \Ham(x) \le E \}
\]
is forward (and backward) invariant: if $x_0 \in K_E$, then
$\Ham(\Phi^t(x_0)) = \Ham(x_0) \le E$ for all $t$, hence $\Phi^t(x_0) \in K_E$.
Theorem~\ref{thm:forward-compact} is applicable when $K_E$ is also compact,
which for example holds under the following common conditions:
\begin{enumerate}
\item[(a)] \textbf{Radially unbounded Hamiltonian.}
If $\Ham(x) \to \infty$ as $\|x\|_2 \to \infty$, then every sublevel set $K_E$ is compact.
Hence $K_E$ is a compact invariant set and the theorem applies.
This situation is typical for mechanical systems of the form
\[
\Ham(q,p) = \tfrac12  p^\top M(q)^{-1} p + V(q),
\]
where $M(q)$ is uniformly positive definite and bounded, i.e., there exist
$0 < m_\ast \le m^\ast < \infty$ such that
$M(q) -
m_\ast I_n$ and $m^\ast I_n - M(q) $ are positive definite for all $q$,
and $V(q) \to \infty$ as $\|q\|_2 \to \infty$.

\item[(b)] \textbf{Bounded configuration space.}
If the system can only move inside a bounded region $\mathcal Q \subset \R^n$
(e.g., due to constraints, periodicity, or walls) and $M(q)$ is uniformly positive
definite and bounded on $\mathcal Q$, then for any energy level $E$ the set
\[
K_E \cap (\mathcal Q \times \R^n)
\]
is compact and invariant under the flow.
\end{enumerate}
\end{remark}

\noindent
Lastly, we demonstrate that, in the quadratic case, a generating function exists for almost every $\Delta T > 0$, thereby removing the step-size restriction.

\begin{theorem}[Existence of a global generating function for quadratic $\Ham$]
\label{thm:quadratic}
Let $\Ham(x) = \tfrac12 x^\top H x$ with a symmetric matrix $H = H^\top \in \R^{2n\times 2n}$ and
\[
\Phi^{\Delta T}(x) = M(\Delta T) x,
\qquad
M(\Delta T) := e^{\Delta T J_{2n} H}
= \begin{pmatrix} A(\Delta T) & B(\Delta T)\\ C(\Delta T) & D(\Delta T) \end{pmatrix}.
\]
Define the resonance set
\[
\mathcal R := \left\{ \Delta T \in \R : \det D(\Delta T) = 0 \right\}.
\]
Then $\mathcal R$ is discrete (it has no finite accumulation point). For any
$\Delta T \notin \mathcal R$, the hypotheses of Theorem~\ref{thm:genInt} hold globally with
$U = \Omega = \R^{2n}$ and $W = \R^{2n}$, and there exists a generating function
$S \in C^1(\R^{2n})$ such that, viewing $p$ and $Q$ as functions of $(q,P)$,
\[
\mathbf p(q,P) = \partial_q S(q,P),\qquad
\mathbf Q(q,P) = \partial_P S(q,P),
\]
and the generating-function identity from Theorem~\ref{thm:genInt} holds on all of $\R^{2n}$.
\end{theorem}

\begin{proof}
The map $\Delta T \mapsto M(\Delta T)$ is real-analytic, being a matrix exponential of
$\Delta T J_{2n}H$. Hence each block $A(\Delta T),B(\Delta T),C(\Delta T),D(\Delta T)$ is
real-analytic in $\Delta T$, and so is
\[
f(\Delta T) := \det D(\Delta T).
\]
We have $M(0) = I_{2n}$, so $D(0) = I_n$ and therefore $f(0) = 1 \neq 0$. Thus $f$ is not
identically zero, and by the identity theorem for real-analytic functions its zero set is
discrete in $\R$. This proves that $\mathcal R = \{ \Delta T : f(\Delta T) = 0 \}$ is discrete.

\noindent Now fix $\Delta T \notin \mathcal R$. Then $D(\Delta T)$ is invertible. Since
\[
\Phi^{\Delta T}(q,p)
= \left(Q,P\right)
= \left(A(\Delta T) q + B(\Delta T) p,\ C(\Delta T) q + D(\Delta T) p\right),
\]
we have
\[
\mathbf P(q,p;\Delta T) = C(\Delta T) q + D(\Delta T) p.
\]
The map
\[
\boldsymbol\Psi : \R^{2n} \to \R^{2n},\qquad
\boldsymbol\Psi(q,p) := \left(q,\mathbf P(q,p;\Delta T)\right)
= \left(q, C(\Delta T) q + D(\Delta T) p\right),
\]
is linear with Jacobian
\[
D\boldsymbol\Psi(q,p)
=
\begin{pmatrix}
I_n & 0\\
C(\Delta T) & D(\Delta T)
\end{pmatrix},
\]
which is invertible because $D(\Delta T)$ is invertible. Consequently, $\boldsymbol\Psi$ is a
global diffeomorphism $\R^{2n} \to \R^{2n}$, and $W := \boldsymbol\Psi(\R^{2n}) = \R^{2n}$ is
simply connected. Thus both conditions (i) and (ii) of Theorem~\ref{thm:genInt} are satisfied
with $U = \Omega = \R^{2n}$ and $W = \R^{2n}$.

\noindent By Theorem~\ref{thm:genInt} there exists $S \in C^1(\R^{2n})$, unique up to an additive constant,
such that, when we express $p$ and $Q$ as functions of $(q,P)$ via
\[
(q,P) = \boldsymbol\Psi(q,p),\qquad (Q,P) = \Phi^{\Delta T}(q,p),
\]
we have
\[
\mathbf p(q,P) = \partial_q S(q,P),\qquad
\mathbf Q(q,P) = \partial_P S(q,P),
\]
and the associated generating function $S^{\Delta T}$ satisfies the generating-function identity
from Theorem~\ref{thm:genInt} on all of $\R^{2n}$.
\end{proof}
\noindent In summary, this section derives verifiable conditions ensuring that the mixed flow increment
\eqref{eqn:mixed-training} admits a type--II generating function $S^{\Delta T}$ on the mixed-data domain
$\Omega$, thereby justifying the RKHS target assumption and enabling the convergence analysis from the
previous section to apply.

\section{Numerical Experiments}\label{Sec:Numerics}

We evaluate the proposed symplectic kernel predictor on benchmark Hamiltonian systems and compare it against a structure-preserving baseline. As a high-fidelity reference we use the implicit midpoint rule with a sufficiently small micro time step $\Delta t = 10^{-3}$. We report two error measures:

\paragraph{(i) maximum residual error versus number of centers.}
Let $s_m$ denote the surrogate after $m$ steps of the greedy procedure, and let
$e_m := u - s_m$ be the corresponding interpolation error of the target potential $u$.
For the training and validation input sets $X_{\mathrm{train}}$ and $X_{\mathrm{val}}$, we report, for
$X \in \{X_{\mathrm{train}},X_{\mathrm{val}}\}$, the discrete uniform norm of the gradient residual
\[
E_X(m)
:=
\max_{\ell\in\mathcal J}\ \max_{x\in X}
\left|\partial_\ell e_m(x)\right|,
\]
i.e., the worst-case gradient mismatch over $X$ as a function of the number of selected centers $m$.

\paragraph{(ii) Relative error over time.}
At multiples of the macro time step $t_k = k\Delta T \in \mathbb{T}_{\Delta T}$ we quantify the deviation from the reference trajectory by
\[
e_{\mathrm{rel}}(t_k)
:=
\frac{\left\|x_{\mathrm{pred}}(t_k) - x_{\mathrm{ref}}(t_k)\right\|_2}
     {\left\|x_{\mathrm{ref}}(t_k)\right\|_2} ,
\]
where $x_{\mathrm{ref}}(t_k)$ denotes the reference solution obtained with the micro-step
implicit midpoint scheme, and $x_{\mathrm{pred}}(t_k)$ is the state obtained from the
symplectic kernel predictor at time $t_k$.

\noindent
Both the kernel and its shape parameter are selected by minimizing the error on the validation set: 
among a grid of shape parameters $\varepsilon > 0$ and the kernels listed below, we choose the pair 
(kernel, $\varepsilon$) that minimizes the validation objective (here, $E_{X_{\mathrm{val}}}(m^\star)$ 
for a fixed $m^\star$). We consider radial kernels of the form $k(x,x';\varepsilon) = \kappa(\varepsilon r)$ 
with $r = \|x - x'\|_2$, including the inverse multiquadric (IMQ) kernel, Gaussian, linear Matérn, 
and quadratic Matérn kernels:
\begin{align*}
\text{IMQ:}\quad
&\kappa_{\mathrm{IMQ}}(r;\varepsilon) = \frac{1}{\sqrt{1 + (\varepsilon r)^2}},\\[0.5ex]
\text{Gaussian:}\quad
&\kappa_{\mathrm{G}}(r;\varepsilon) = \exp \left(-(\varepsilon r)^2\right),\\[0.5ex]
\text{Matérn-$\tfrac{3}{2}$ (``linear'' Matérn):}\quad
&\kappa_{\mathrm{M1}}(r;\varepsilon) = \left(1 + \varepsilon r\right)\exp \left(-\varepsilon r\right),\\[0.5ex]
\text{Matérn-$\tfrac{5}{2}$ (``quadratic'' Matérn):}\quad
&\kappa_{\mathrm{M2}}(r;\varepsilon) = \left(1 + \varepsilon r + \tfrac{1}{3}(\varepsilon r)^2\right)\exp \left(-\varepsilon r\right).
\end{align*}

\subsection{Pendulum}\label{sec:pendulum}
We first consider the mathematical pendulum with Hamiltonian
\[
  \mathcal H(q,p)
  = \frac{p^2}{2 m l^{2}} + m g l \left(1-\cos q\right),
  \qquad m = l = 1,\; g = 9.81,
\]
so that $\mathcal H(q,p) = \tfrac12 p^2 + g \left(1-\cos q\right)$ and choose a final time $T = 6.0$. We define the box
\[
\hat \Omega := [-\pi,\pi]\times\left[-2\sqrt{g}, 2\sqrt{g}\right],
\]
and the energy-bounded domain
\[
  \Omega := \left\{(q,p)\in\hat \Omega:\ \mathcal H(q,p) < 2g\right\}.
\]

\noindent
For each initial state $x_0 = (q_0,p_0)$ we compute the reference solution
$x(\cdot;x_0)$ with the implicit midpoint rule using the micro time step $\Delta t = 10^{-3}$.
As a structure-preserving baseline for large steps we apply the implicit midpoint rule directly
with the macro step $\Delta T$.

\noindent We consider two cases that differ in the sampling scenario. In both cases, we first create
a uniform tensor-product grid $X$ on $\hat \Omega$ with $200\times 200$ points and retain only
those inside the energy-bounded set $\Omega$. We then define the training sets
\[
    \mathcal M_{\mathrm{A}}^{\Delta T}
    := \left\{\left(I_1 x_0 + I_2\Phi^{\Delta T}(x_0), J_{2n}^{\top} \frac{\Phi^{\Delta T}(x_0)-x_0}{\Delta T}\right)
     : x_0 \in X \cap \Omega \right\},
\]
and
\[
    \mathcal M_{\mathrm{B}}^{\Delta T}
    := \left\{\left(I_1 x_0 + I_2\Phi^{\Delta T}(x_0), J_{2n}^{\top} \frac{\Phi^{\Delta T}(x_0)-x_0}{\Delta T}\right): x_0 \in X \cap \Omega^- \right\},
\]
with
\[
\Omega^- := \left\{(q,p)\in\Omega : p \le 0\right\}.
\]
In both cases, $x(\Delta T;x_0)$ is approximated using the implicit midpoint rule with the micro
time step $\Delta t$.

\noindent
For testing, we draw $10$ initial conditions with $q_0 \sim \mathcal U([0,\pi])$ and $p_0 = 0$,
and evolve each trajectory up to $T = 6.0$ using
(i) the proposed symplectic predictor,
(ii) the implicit midpoint rule with step $\Delta T$, and
(iii) the reference micro-step solution (time step $\Delta t$). Note that
\[
\nabla^2\Ham(q,p)=\begin{pmatrix}9.81\cos q&0\\[2pt]0&1\end{pmatrix},
\qquad
L_K:=\sup_{(q,p)\in K}\|\nabla^2\Ham(q,p)\|_2\le 9.81,
\]
so that
\[
\Delta T^\star_K=\min\left\{T,\ \tfrac{\log 2}{9.81}\right\}\approx 7.07\times 10^{-2}.
\]
Hence, the hypotheses of \Cref{thm:forward-compact} are guaranteed for $\Delta T<7.07\times 10^{-2}$, implying existence of a generating function; since the underlying Grönwall-type estimate is conservative, we test $\Delta T\in\{0.1, 0.05, 0.025\}$, including a regime with $\Delta T>\Delta T^\star_K$.

\noindent As a first numerical result, we present the convergence of the greedy procedure in
\Cref{fig:fgreedy-convergence}.

\begin{figure}[h]
  \centering
  \begin{subfigure}[t]{0.64\textwidth}
    \centering
    \includegraphics[width=\linewidth]{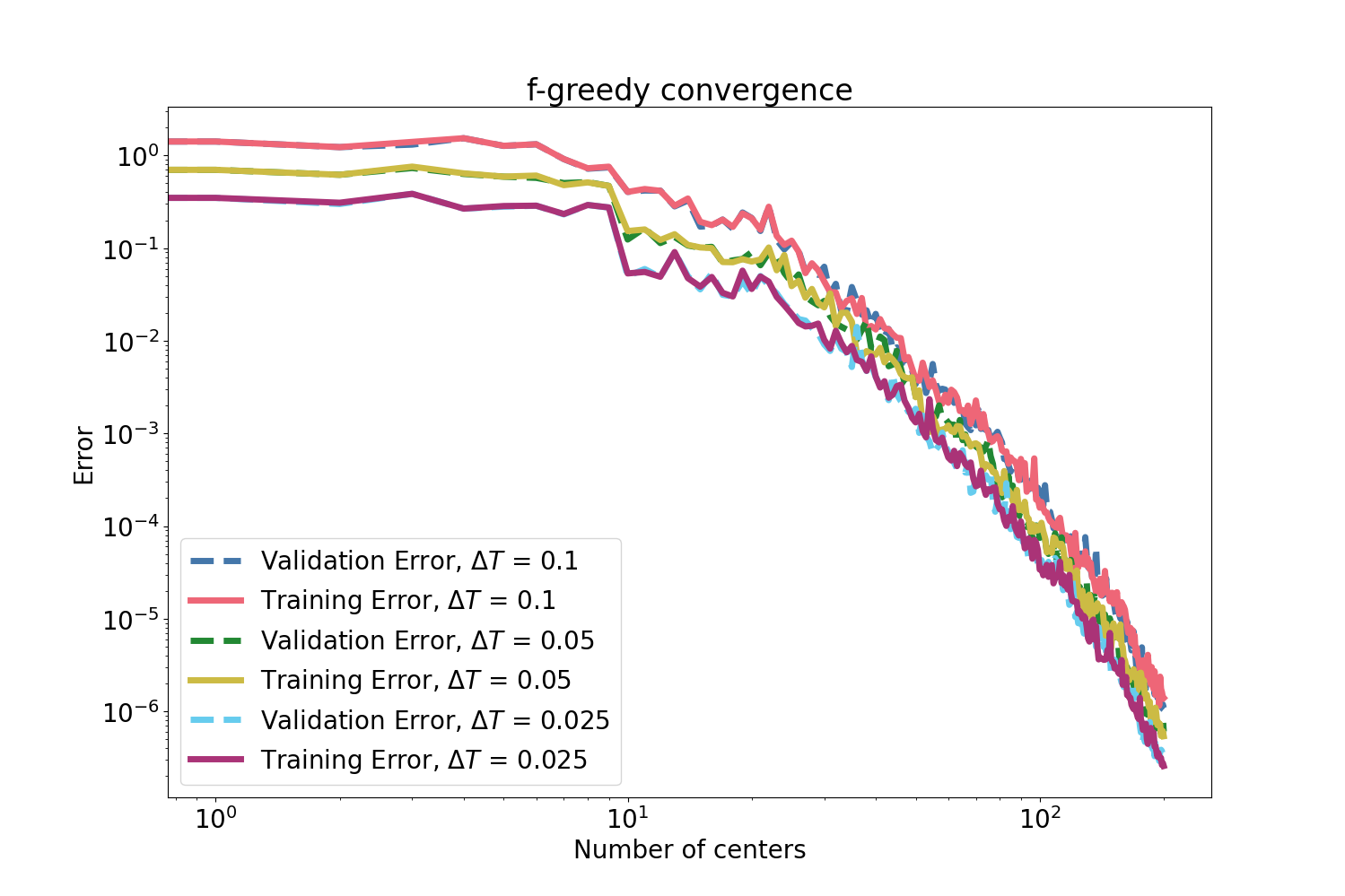}
    \caption{$f$-greedy interpolation error versus the number of selected centers for three macro time step sizes, showing training (solid) and validation (dashed) curves.}
    \label{fig:fgreedy-convergence}
  \end{subfigure}\hfill
  \begin{subfigure}[t]{0.35\textwidth}
    \centering
    \includegraphics[width=\linewidth]{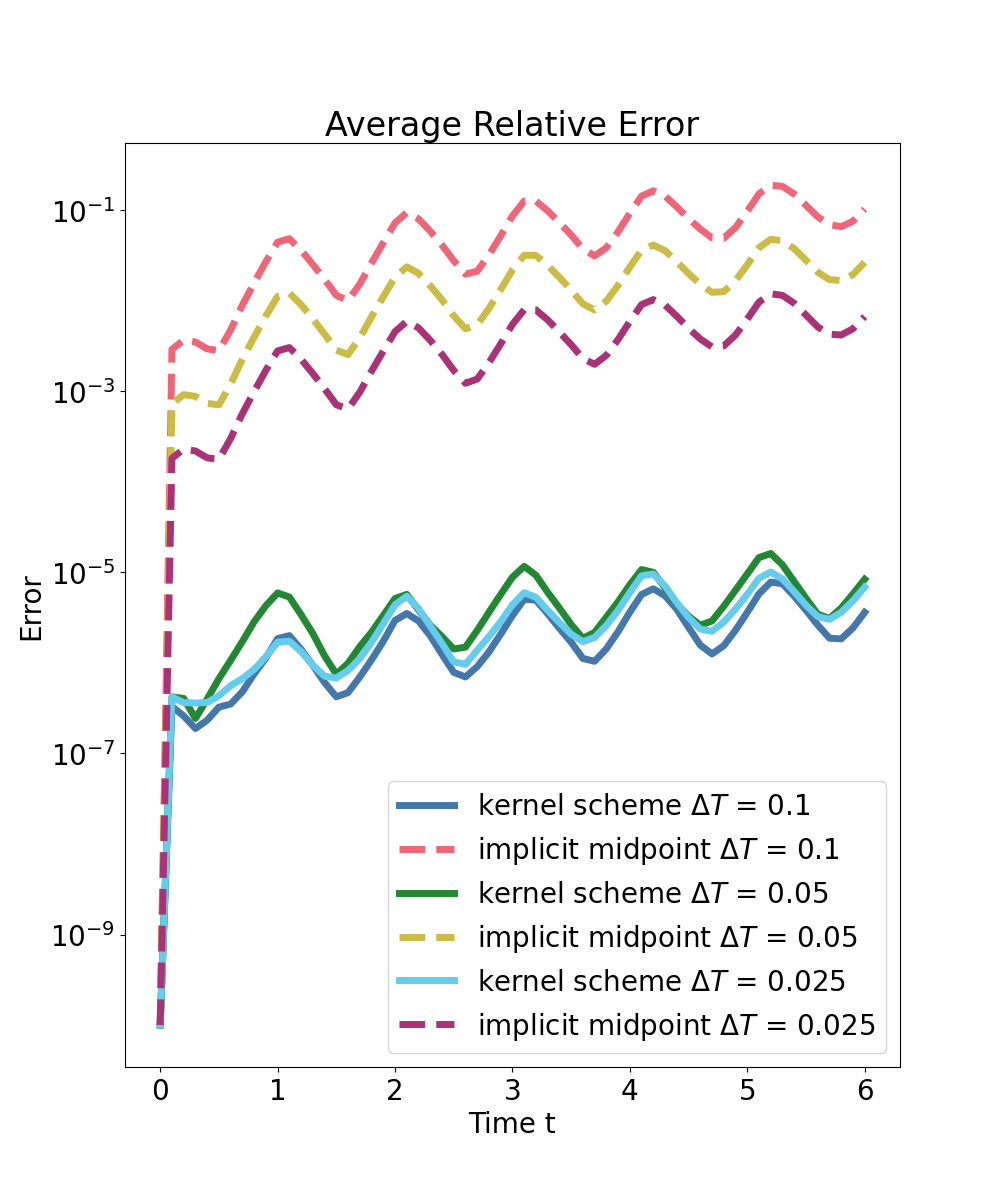}
    \caption{Relative error over time comparing the symplectic kernel scheme (solid) with the implicit-midpoint baseline (dashed) for three macro time step sizes.}
    \label{fig:rel-reduction-error}
  \end{subfigure}
  \caption{Pendulum: (a) $f$-greedy convergence vs.\ centers; (b) relative error over time.}
\end{figure}

\noindent
For
$\Delta T \in \{0.1,0.05,0.025\}$, we report the maximum residual error as a function of the number of centers  both for training and validation data. After a short plateau (for roughly the first $15$ centers),
all curves exhibit an almost straight-line decay on the log--log axes, i.e., an approximately
algebraic convergence, reaching values of order $\sim 10^{-6}$ by about $200$ centers.
Training and validation errors track closely for all $\Delta T$, with the validation error
slightly above the training error, indicating good generalization and no visible overfitting.
Moreover, the error decay for smaller $\Delta T$ starts from a slightly lower level and
attains a slightly lower minimal error.

\noindent
Next, we apply the kernel scheme to predict the $10$ trajectories corresponding to the
$10$ randomly chosen initial conditions described above, and present the average relative
error. The results are shown in \Cref{fig:rel-reduction-error}. Over the entire time
horizon, the kernel scheme achieves errors between $\sim 10^{-7}$ and $\sim 10^{-5}$.
The implicit-midpoint curves are $3$--$4$ orders of magnitude larger at the same $\Delta T$,
rising from $\sim 10^{-3}$ to $\sim 10^{-1}$ while displaying similar oscillations.
Decreasing $\Delta T$ uniformly lowers all implicit-midpoint curves (best for $\Delta T = 0.025$,
then $0.05$, then $0.1$); however, the kernel predictor remains more accurate for every step size.
In contrast and remarkably, even though the kernel scheme exhibits slightly faster convergence for smaller
$\Delta T$, it performs best for the largest $\Delta T$, likely due to the smaller number of
macro steps required to span the trajectory, which reduces error accumulation.

\noindent
Next, we show the same experiment, but with a reduced training set $\mathcal M_{\mathrm{B}}^{\Delta T}$. We reduce the maximum number of centers by a factor of $1/2$, yielding a fill distance comparable to that in the previous experiment with the larger domain.

\begin{figure}[h]
  \centering
  \begin{subfigure}[t]{0.64\textwidth}
     \centering
     \includegraphics[width = \textwidth]{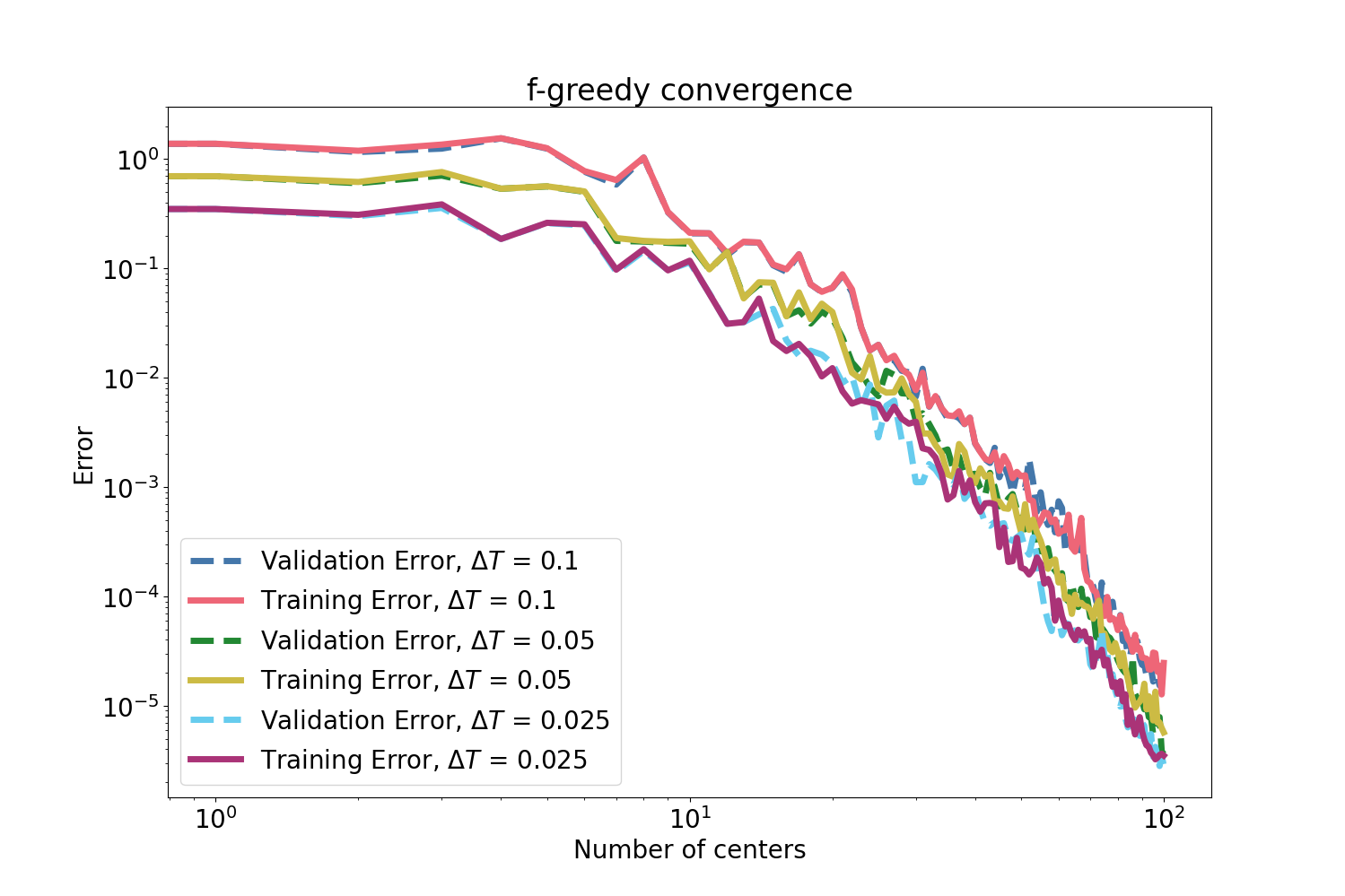}
     \caption{$f$-greedy interpolation error versus the number of selected centers for three macro time step sizes, showing training (solid) and validation (dashed) curves.}
     \label{fig:fgreedy-convergence-gen}
  \end{subfigure}\hfill
  \begin{subfigure}[t]{0.35\textwidth}
    \centering
    \includegraphics[width = \textwidth]{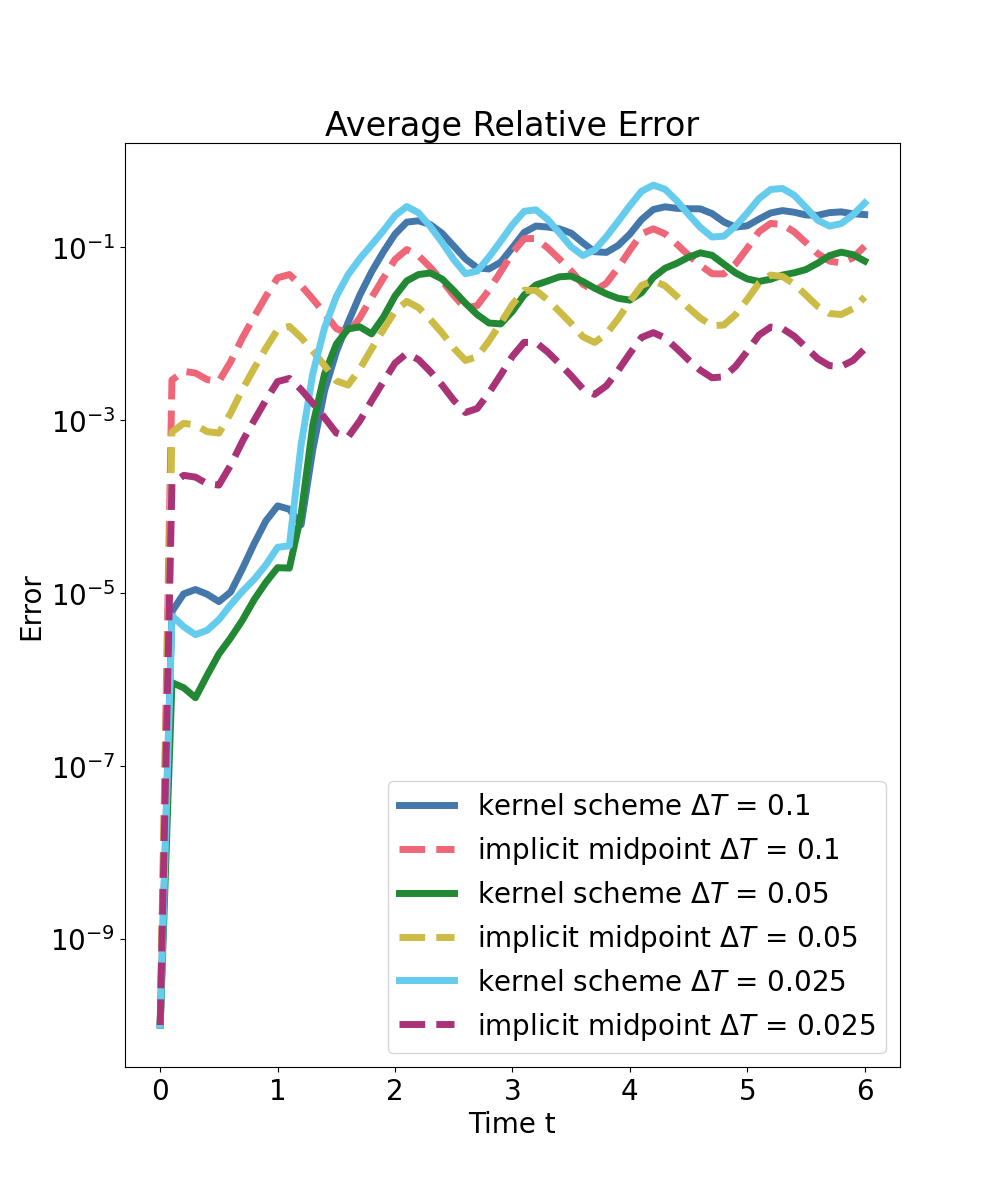}
    \caption{Relative error over time comparing the symplectic kernel scheme (solid) with the implicit-midpoint baseline (dashed) for three macro time step sizes.}
    \label{fig:rel-reduction-error-gen}
  \end{subfigure}
  \caption{Pendulum (reduced training set): (a) $f$-greedy convergence vs.\ centers; (b) relative error over time.}
\end{figure}

\noindent
In \Cref{fig:fgreedy-convergence-gen} we report the maximum residual error as a function of the number
of centers for $\Delta T \in \{0.1,0.05,0.025\}$. The decay is very similar to the original
experiment with the larger training set.

\noindent
In \Cref{fig:rel-reduction-error-gen}, the relative error is shown, averaged over the
same $10$ random initial conditions. Over the horizon $T = 6$, the kernel scheme starts with very
small errors (about $10^{-7}$–$10^{-5}$ near $t \approx 0$) but then these increase as the
trajectory leaves the training domain at $t = \pi\sqrt{l/g} \approx 1.00$ and the error settles
between $10^{-2}$ and $10^{-1}$, growing only mildly over time. Beyond the training horizon
($t \approx 1.00$), the error of the symplectic kernel predictor remains bounded, with the small
oscillations typical of symplectic schemes. For $\Delta T = 0.05$, its accuracy is comparable to
that of the implicit midpoint scheme with $\Delta T = 0.1$ up to $T = 6$. Note that these results still consistently indicate very good performance of all the symplectic schemes, as a non-symplectic integrator, e.g., explicit Euler often produces relative error several orders of magnitude larger than 100\% \cite{hairer2006,hairer2012numerical,chyba2009role,csillik2004symplectic}.

\noindent
To highlight the generalization capabilities, in \Cref{fig:pendulum-gen-trajectories} we compare
one pendulum trajectory obtained from the kernel scheme (solid blue, $\Delta T = 0.025$) with the
reference solution (orange dashed), for both the angle $q(t)$ (top panel) and the angular momentum
$p(t)$ (bottom panel).
\begin{figure}[h]
    \centering
    \includegraphics[width=0.7\linewidth]{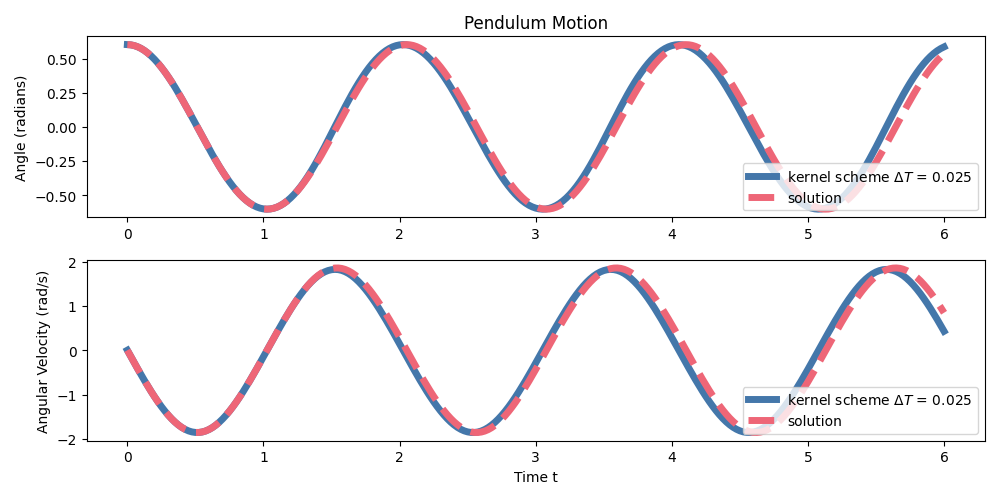}
    \caption{Pendulum generalization test: comparison of one trajectory obtained from the kernel scheme ($\Delta T = 0.025$) with the reference solution.}
    \label{fig:pendulum-gen-trajectories}
\end{figure}
\noindent
The two curves almost coincide over the entire time interval, demonstrating very good
generalization beyond the training domain: the amplitudes and qualitative shape of the oscillations
are preserved, and no spurious growth or unphysical behavior is observed. The remaining discrepancy
is dominated by a slight phase shift between predictor and reference, rather than by an incorrect
or nonphysical trajectory.

\subsection{Nonlinear spring--mass chain with fixed ends}

We next consider a nonlinear spring--mass chain, that is, a row of point masses connected by springs between two fixed walls. Let $n\in\mathbb{N}$ point masses be arranged on a line with displacements 
\[
q=(q_1,\dots,q_n)^\top\]
that are measured from the equilibrium configuration,
momenta
\[ p=(p_1,\dots,p_n)^\top,
\]
and mass matrix $M=\mathrm{diag}(m_1,\dots,m_n)$. In the experiments below we take $m_1=\dots=m_n=1$, so that $p = M\dot q = \dot q$. Fixed ends are enforced via virtual nodes $q_0\equiv 0$ and $q_{n+1}\equiv 0$. The spring elongations are given by
\[
\delta_i(q):=q_{i+1}-q_i,\qquad i=0,\dots,n,
\]
which can be written compactly as $\delta = Bq$ with
\[
B =
\begin{bmatrix}
  1 & 0 & 0 & \cdots & 0 \\
 -1 & 1 & 0 & \cdots & 0 \\
  0 & -1 & 1 & \ddots & \vdots \\
  \vdots & \ddots & \ddots & \ddots & 0 \\
  0 & \cdots & 0 & -1 & 1 \\
  0 & \cdots & \cdots & 0 & -1
\end{bmatrix} \in \mathbb{R}^{(n+1)\times n},
\]
so that
\[
(Bq)_0 = q_1,\qquad (Bq)_i = q_{i+1}-q_i\ (1\le i\le n-1),\qquad (Bq)_n = - q_n,
\]
and hence $Bq=(\delta_0,\dots,\delta_n)^\top$.

\noindent
Given $C^1$ spring potentials $f_i:\mathbb{R}\to\mathbb{R}$ and
\[
\sigma(z):=\left(f_0'(z_0),\dots,f_n'(z_n)\right)^\top,\qquad z=(z_i)_{i=0}^n\in\mathbb{R}^{n+1},
\]
the potential energy and its gradient are
\[
V(q)=\sum_{i=0}^n f_i\left((Bq)_i\right),\qquad
\nabla_q V(q)=B^\top\sigma(Bq).
\]
We consider identical quartic spring potentials of the form
\[
f_i(\delta_i)=\tfrac12 \alpha \delta_i^2+\tfrac{\beta}{4} \delta_i^4,
\qquad
\sigma_i(\delta_i)=f_i'(\delta_i)=\alpha \delta_i+\beta \delta_i^3,
\]
with parameters $\alpha=1$ and $\beta=0.25$ for all $i=0,\dots,n$. In the absence of damping, the equations of motion read
\[
M\ddot q + B^\top\sigma(Bq)=0.
\]
Introducing the state $x=(q,p)$ and the Hamiltonian
\[
\mathcal H(q,p)=\tfrac12 p^\top M^{-1}p + V(q),
\]
we obtain the canonical Hamiltonian system
\[
\dot x = J_{2n} \nabla H(x),
\]
with the standard symplectic matrix $J_{2n}$. In the numerical experiments below we fix $n=3$.

\noindent
We study two training scenarios that differ in the sampling of initial conditions. In both cases we first draw candidate states from a bounded box and retain only those below a prescribed energy level. Specifically, let
\[
\hat\Omega=\left([-q_{\max}, q_{\max}]^{ n}\times[-p_{\max}, p_{\max}]^{ n}\right)\subset\mathbb{R}^{2n},
\qquad
q_{\max}=0.5,\ \ p_{\max}=0.5,
\]
and define the energy-bounded subset
\[
\Omega=\left\{(q,p)\in\hat\Omega:\ \mathcal H(q,p)\le H_{\max}\right\},
\qquad
\mathcal H(q,p)=\tfrac12 p^\top M^{-1}p+\sum_{i=0}^{n} f_i\left(\delta_i(q)\right),
\quad H_{\max}=0.5 .
\]
We generate states
\[
(q^{(j)},p^{(j)})\sim\mathcal{U}(\hat\Omega),
\]
and retain only those satisfying the energy constraint:
\[
X=\left\{(q^{(j)},p^{(j)})\in\hat\Omega:\ H(q^{(j)},p^{(j)})\le H_{\max}\right\}\subset\Omega
\]
until $N_s = 10000$ states have been collected. 

\noindent
The training data for scenario~A are defined as
\[
\mathcal M_{\mathrm{A}}^{\Delta T}
    := \left\{\left(I_1 x_0 + I_2\Phi^{\Delta T}(x_0), J_{2n}^{\top} \frac{\Phi^{\Delta T}(x_0)-x_0}{\Delta T}\right):\ x_0\in X \right\},
\]
that is, pairs of states at times $0$ and $\Delta T$ along trajectories starting from $x_0\in X$. In scenario~B, we restrict the initial conditions to a half-space in momentum,
\[
\Omega^- = \left\{(q,p)\in\Omega:\; p_2\le 0\right\},
\]
and construct
\[
\mathcal M_{\mathrm{B}}^{\Delta T}
    := \left\{\left(I_1 x_0 + I_2\Phi^{\Delta T}(x_0), J_{2n}^{\top} \frac{\Phi^{\Delta T}(x_0)-x_0}{\Delta T}\right) :\ x_0\in X\cap \Omega^- \right\}.
\]
In both cases, $x(\Delta T;x_0)$ is approximated by the implicit midpoint rule with micro time step $\Delta t$.

\noindent
For evaluation we draw $10$ test initial conditions independently of the training data according to
\[
q_0\sim\mathcal{U} \left([0, q_{\max}]^3\right),\qquad p_0=0,
\]
and integrate each test trajectory up to $T=10.0$ using
(i) the proposed symplectic kernel predictor,
(ii) the implicit midpoint method with macro time step $\Delta T$, and
(iii) a high-fidelity reference solution obtained by the implicit midpoint rule with micro time
step $\Delta t$.

\noindent Note, that we have
\[
\nabla^2\Ham(q,p)=
\begin{pmatrix}
B^\top \textrm{diag}  \left(1+3\beta (Bq)_i^2\right)_{i=0}^n B & 0\\[2pt]
0 & I_n
\end{pmatrix},
\qquad
L_K\le \max\left\{1,\ \|B\|_2^2\left(1+3\beta \delta_K^2\right)\right\},
\]
with $\delta_K:=\sup_{(q,p)\in K}\|Bq\|_\infty$. \noindent For $n=3$ and $q_i\in[-0.5,0.5]$, each elongation is a difference of two coordinates (or a boundary value),
\[
\delta_0=q_1,\qquad \delta_1=q_2-q_1,\qquad \delta_2=q_3-q_2,\qquad \delta_3=-q_3,
\]
hence $|\delta_i|\le 1$ for all $i$ and therefore $\delta_K=\sup_{(q,p)\in K}\|Bq\|_\infty\le 1$.
Moreover,
\[
\|B\|_2^2=\|B^\top B\|_2=\lambda_{\max}(B^\top B).
\]
Since $B^\top B$ is tridiagonal with diagonal entries $2$ and off-diagonals $-1$, the Gershgorin
circle theorem yields $\sigma(B^\top B)\subset[0,4]$, hence
\[
\|B^\top B\|_2=\lambda_{\max}(B^\top B)\le 4
\qquad\text{and}\qquad
\|B\|_2^2\le 4.
\]
Therefore,
\[
L_K \le 4\left(1+\tfrac34\right)=7,
\qquad
\Delta T_K^\star=\min\left\{T,\ \tfrac{\log 2}{L_K}\right\}\ \ge\ \tfrac{\log 2}{7}
\approx \ 9.90\times 10^{-2}.
\]

\begin{figure}[h]
  \centering
  \begin{subfigure}[t]{0.64\textwidth}
    \centering
    \includegraphics[width=\textwidth]{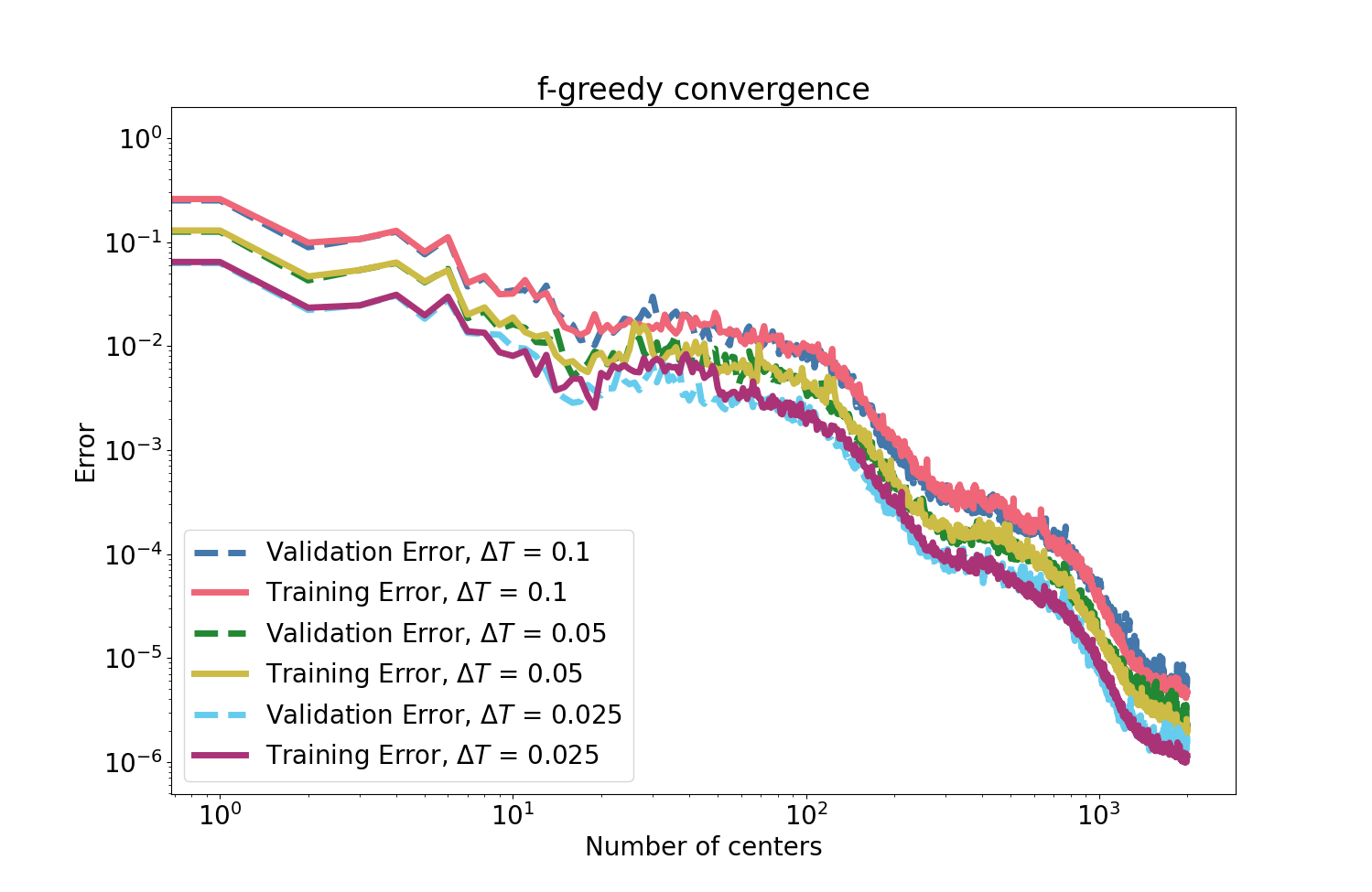}
    \caption{$f$-greedy interpolation error versus the number of selected centers (mass--spring chain), training (solid) and validation (dashed).}
    \label{fig:fgreedy-convergence-mass-spring}
  \end{subfigure}\hfill
  \begin{subfigure}[t]{0.35\textwidth}
    \centering
    \includegraphics[width=\textwidth]{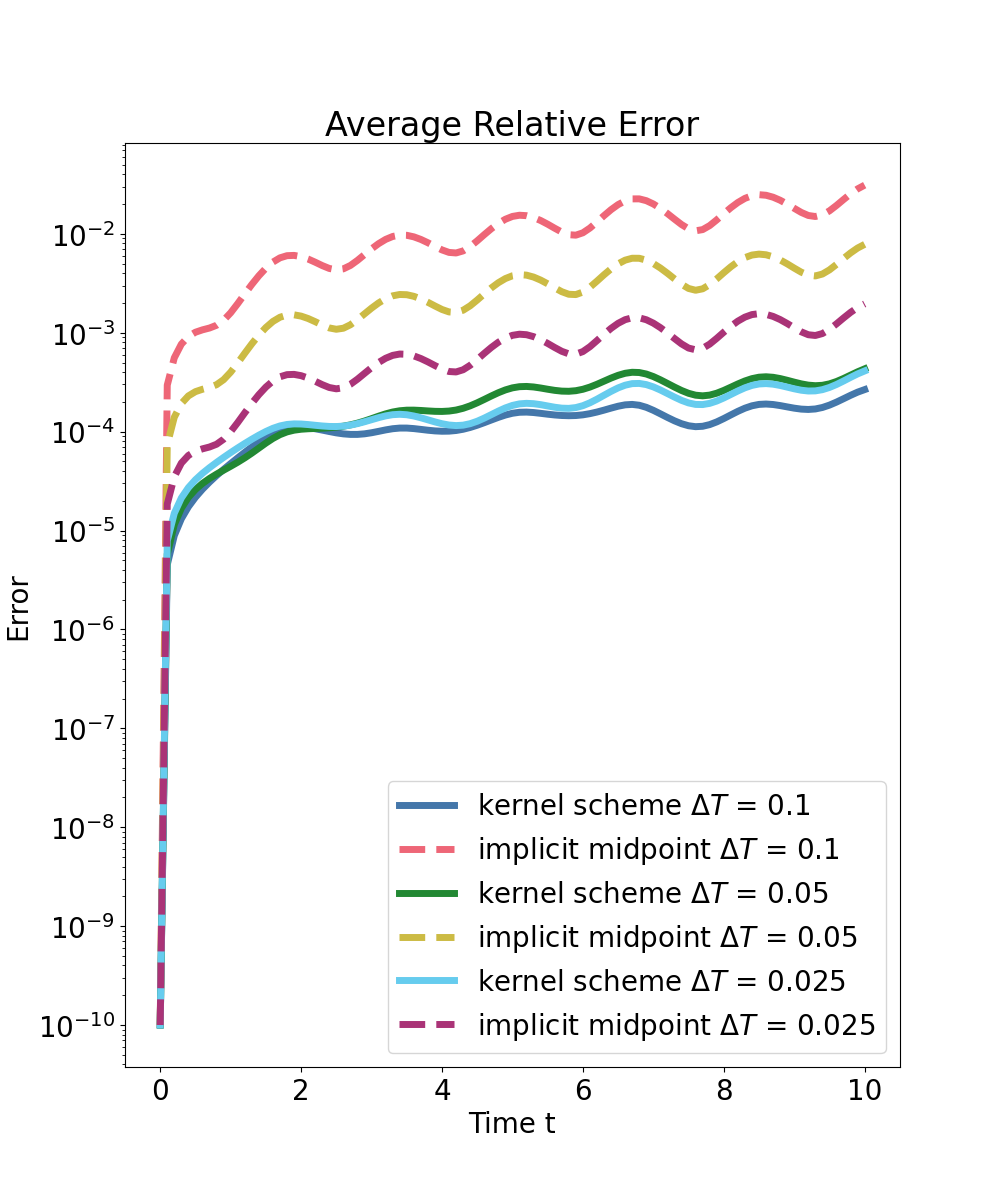}
    \caption{Relative error over time (mass--spring chain), kernel predictor (solid) vs.\ implicit midpoint (dashed).}
    \label{fig:rel-reduction-error-mass-spring}
  \end{subfigure}
  \caption{Mass--spring chain: (a) $f$-greedy convergence vs.\ centers; (b) relative error over time.}
\end{figure}

\noindent
As a first numerical result we report in \Cref{fig:fgreedy-convergence-mass-spring} the convergence
of the greedy center selection. The maximum residual error is shown as a function of the number of centers
$m$ for $\Delta T \in \{0.1,0.05,0.025\}$. For all three macro time steps, the error decreases
steadily from approximately $10^{-1}$ to about $10^{-3}$ around $m \approx 10^2$, and reaches values
close to $10^{-6}$ by $m \approx 10^3$. Training and validation curves are very close over the
entire range of $m$, with the validation error only slightly above the training error, indicating
good generalization and no apparent overfitting. As in the previous examples, smaller macro time
steps yield slightly lower error levels, but the qualitative decay behavior is essentially the same
for all $\Delta T$.

\noindent
We then apply the kernel scheme to the $10$ test trajectories and show the average relative
error in \Cref{fig:rel-reduction-error-mass-spring}. Over the full time interval
$[0,10]$, the kernel predictor achieves errors between approximately $10^{-5}$ and a few
$\times 10^{-4}$, with only mild temporal oscillations. The structure-preserving baseline (implicit
midpoint with step size $\Delta T$) is systematically less accurate: its error ranges from about
$10^{-3}$ to a few $\times 10^{-2}$, i.e., typically one to almost two orders of magnitude larger
than that of the kernel predictor at the same $\Delta T$, while displaying a similar oscillatory
pattern. Reducing $\Delta T$ lowers all implicit-midpoint curves as expected, but the kernel
predictor remains more accurate for every step size and throughout the entire time interval.

\noindent
We next repeat the experiment with the reduced training set $\mathcal M_{\mathrm{B}}^{\Delta T}$ (and a maximum of $1000$ centers); see
\Cref{fig:fgreedy-convergence-mass_spring-gen,fig:rel-reduction-error-mass-spring-gen}.

\begin{figure}[h]
  \centering
  \begin{subfigure}[t]{0.64\textwidth}
    \centering
    \includegraphics[width=\textwidth]{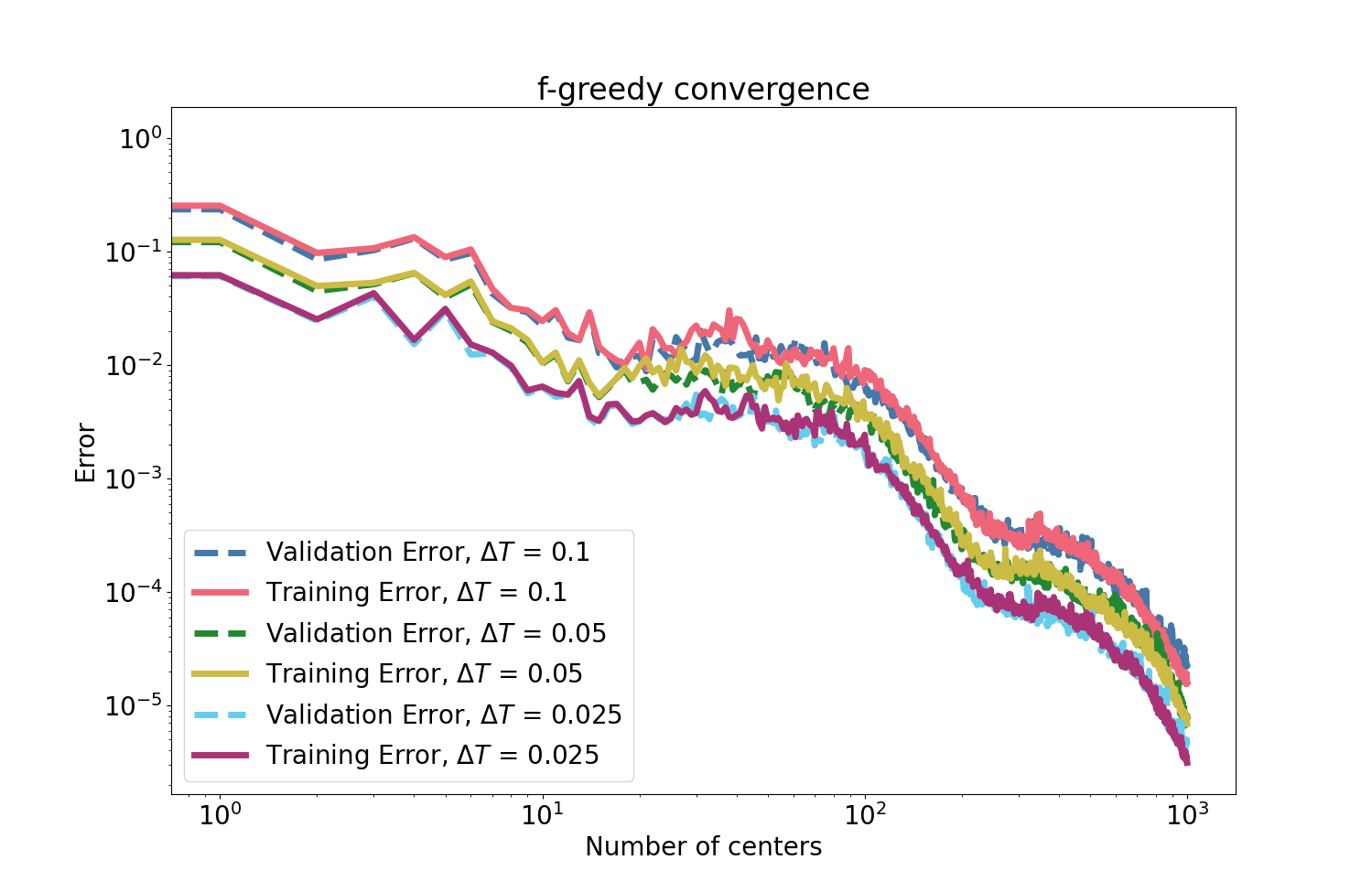}
    \caption{$f$-greedy interpolation error versus the number of selected centers (mass--spring chain, reduced training set), training (solid) and validation (dashed).}
    \label{fig:fgreedy-convergence-mass_spring-gen}
  \end{subfigure}\hfill
  \begin{subfigure}[t]{0.35\textwidth}
    \centering
    \includegraphics[width=\textwidth]{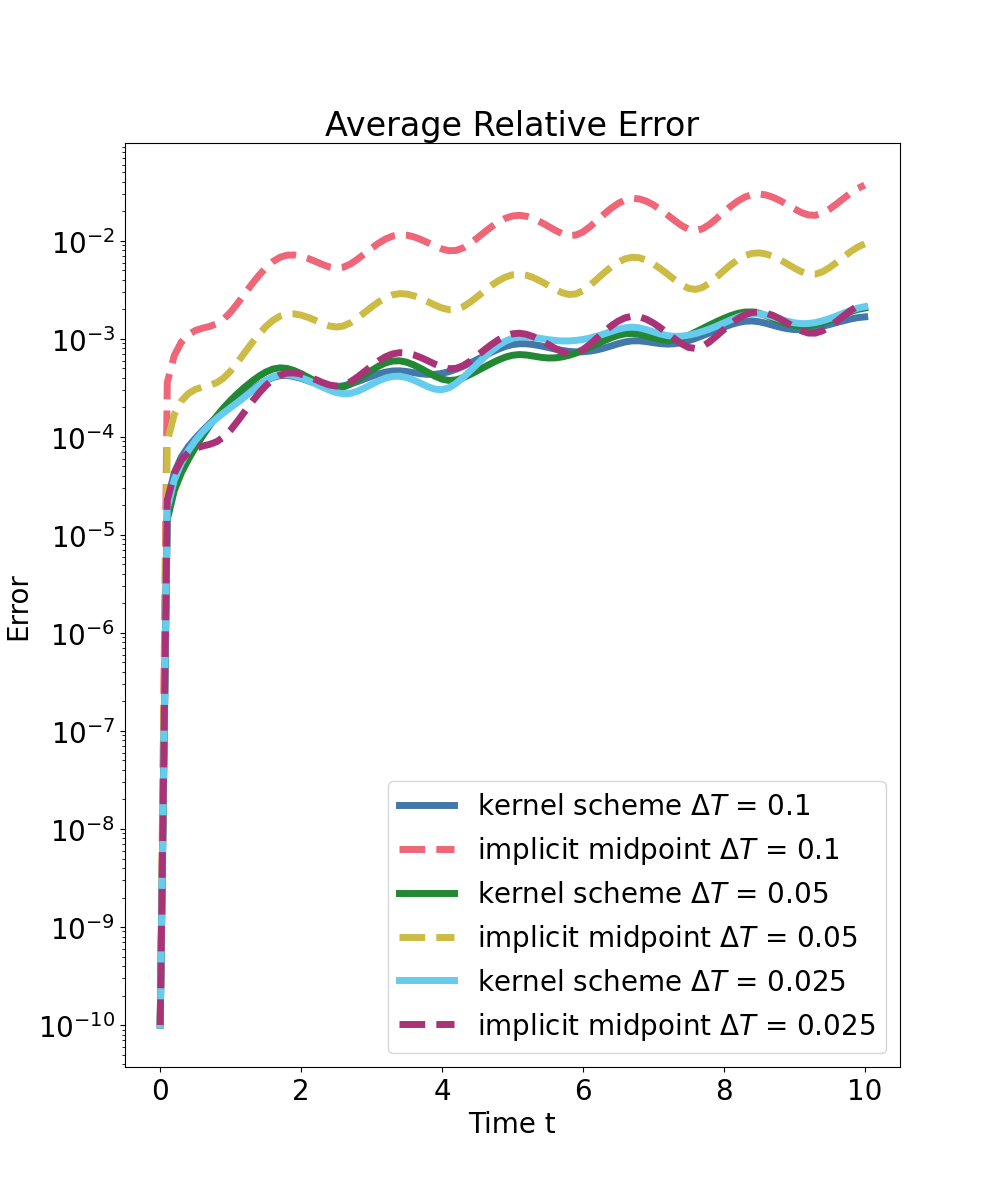}
    \caption{Relative error over time (mass--spring chain, reduced training set), kernel predictor (solid) vs.\ implicit midpoint (dashed).}
    \label{fig:rel-reduction-error-mass-spring-gen}
  \end{subfigure}
  \caption{Mass--spring chain (reduced training set): (a) $f$-greedy convergence vs.\ centers; (b) relative error over time.}
\end{figure}
\noindent
In \Cref{fig:fgreedy-convergence-mass_spring-gen}, the maximum residual error again decreases
monotonically with the number of centers for all $\Delta T \in \{0.1,0.05,0.025\}$. The curves
start between approximately $10^{-1}$ and $3\times 10^{-1}$ and drop to the $10^{-3}$ level around
$m \approx 100$, reaching about $10^{-6}$ by $m \approx 10^3$. Training and validation errors
remain close throughout the range of $m$, with the validation curves only slightly above the
training curves, again indicating good generalization and no visible overfitting. As before,
smaller macro time steps yield lower error levels, but the decay behavior is essentially identical
for all three $\Delta T$.

\noindent
In \Cref{fig:rel-reduction-error-mass-spring-gen}, the average relative error over
$10$ random initial conditions remains well controlled on $[0,10]$, with the kernel scheme
achieving errors between roughly $10^{-5}$ and $10^{-3}$ depending on $\Delta T$. The
implicit-midpoint baseline is consistently less accurate, with errors between about $10^{-4}$ and
a few $\times 10^{-2}$, i.e., typically one to almost two orders of magnitude larger at matched
$\Delta T$, while again exhibiting similar oscillatory patterns. The kernel predictor
generalizes well along the entire trajectories: its error does not exhibit uncontrolled growth and
remains below the baseline for most of the time. For each $\Delta T $, the kernel scheme is
approximately as accurate as the implicit midpoint rule with the smallest macro time step
$\Delta T = 0.025$ over the full interval up to $T = 10$.

\noindent
The generalization experiment for the spring--mass chain (training on
$\Omega^-=\{(q,p)\in\Omega: p_2\le 0\}$ and testing on $p_2>0$) performs almost as well as the
fully sampled case. This behavior can be related to the mechanical Hamiltonian structure
$\mathcal H(q,p)=\tfrac12 p^\top M^{-1}p+V(q)$, which depends on $p$ only through the quadratic kinetic
energy and is therefore invariant under $p\mapsto -p$. The associated Hamiltonian vector field
changes sign when $p$ is reversed. Consequently, the ``unseen'' half-space $\{p_2>0\}$ essentially
contains the same trajectories as $\{p_2<0\}$, traversed with opposite momentum. In addition, the
symplectic Euler argument $\xi_j=(q_0^{(j)},p_{\Delta T}^{(j)})$ already contains many samples with positive
$p_{\Delta T,2}$ even when $p_{0,2}\le 0$, and the learning targets are difference quotients of
this sign-reversing vector field. Combined with the weakly nonlinear, small-amplitude regime
considered here, these symmetries make extrapolation across $p_2=0$ comparatively benign and help
explain the strong generalization performance.

\subsection{Discretized wave equation}
We consider the one-dimensional wave equation on a finite interval. For time $t \in I = (0,T)$
and spatial variable $\xi \in \Omega := (0,L)$, find $u : \bar I \times \bar \Omega \to \R$ such that
\begin{align*}
u_{tt}(t,\xi) &= c^2 u_{\xi\xi}(t,\xi) && \text{in } I\times\Omega,\\
u(t,\xi) &= 0 && \text{on } I\times\partial\Omega,\\
u(0,\xi) &= u_0(\xi),\qquad
u_t(0,\xi) = v_0(\xi) && \text{in } \Omega,
\end{align*}
with wave speed $c = 0.3$ domain length $L = 1$ and end time $T = 6$ in the numerical experiments.

\noindent
We discretize $\Omega$ by a uniform grid $\{\xi_i\}_{i=1}^N \subset (0,L)$ with $N = 1000$ and enforce homogeneous
Dirichlet data at $\xi = \{0,L\}$. Let $D_{\xi\xi} \in \R^{N\times N}$ be the standard (symmetric
positive definite) central-difference matrix for $-\partial_{\xi\xi}$ with homogeneous Dirichlet
boundaries. With $u(t) \in \R^N$ the vector of nodal values, the semi-discrete equation reads
\[
\ddot u(t) = -c^2 D_{\xi\xi} u(t).
\]
Introduce the phase–space state
\[
x(t)=\begin{bmatrix} q(t) \\ p(t) \end{bmatrix}
=\begin{bmatrix} u(t) \\ \dot u(t) \end{bmatrix}\in\R^{2N},
\]
and the quadratic Hamiltonian
\[
\Ham(x)=\tfrac12 p^\top p+\tfrac12 c^2 q^\top D_{\xi\xi} q.
\]
\noindent Then the dynamics take canonical Hamiltonian form
\[
\dot x(t)
= J_{2N} \nabla_x \Ham\big(x(t)\big)
= J_{2N} H x(t),\qquad
J_{2N}=\begin{bmatrix}0&I_N\\ -I_N&0\end{bmatrix},
\]
with the symmetric matrix
\[
H=\begin{bmatrix}c^2 D_{\xi\xi} & 0\\[2pt] 0 & I_N\end{bmatrix},
\qquad
x(0)=\big[u_0(\xi_1),\dots,u_0(\xi_N),\ v_0(\xi_1),\dots,v_0(\xi_N)\big]^\top.
\]
\noindent Note that, since $\Ham$ is quadratic, the hypotheses of \Cref{thm:quadratic} apply; consequently, a (global) generating function exists for all step sizes $\Delta T$ outside a discrete resonance set.

\noindent 
To compress the $2N$-dimensional Hamiltonian system while preserving structure, we use symplectic
model order reduction \cite{Peng2016,Buchfink2019,afkham2017structure} to project onto a symplectic
subspace spanned by $V \in \R^{2N\times 2n}$, $n \ll N$, satisfying $V^\top J_{2N} V = J_{2n}$.
Reduced coordinates are obtained with the symplectic inverse
\[
x_{\mathrm{red}}(t) = V^{+}x(t),\qquad
V^{+}:=J_{2n}^\top V^\top J_{2N},\qquad
x(t)\approx V x_{\mathrm{red}}(t).
\]
For the quadratic Hamiltonian $\Ham(x)=\tfrac12 x^\top H x$ the reduced Hamiltonian is
\[
\Ham_{\mathrm{red}}(x_{\mathrm{red}})
=\tfrac12 x_{\mathrm{red}}^\top \big(V^\top H V\big) x_{\mathrm{red}},
\]
and the reduced dynamics remain a canonical Hamiltonian system:
\[
\dot x_{\mathrm{red}}(t)
= J_{2n} \nabla_{x_{\mathrm{red}}} \Ham_{\mathrm{red}}(x_{\mathrm{red}})
= J_{2n} \big(V^\top H V\big) x_{\mathrm{red}}(t),\qquad
x_{\mathrm{red}}(0)=V^{+}x(0).
\]

\noindent
In practice, we construct $V$ via the complex SVD (cSVD): collect snapshots
$Q=[q(t_k)]$ and $P=[p(t_k)]$, form the complex snapshot matrix $Y := Q + \mathrm{i}P$, compute
the thin SVD $Y \approx U\Sigma W^*$, and set
\[
V=
\begin{bmatrix}
\operatorname{Re}U & -\operatorname{Im}U\\[2pt]
\operatorname{Im}U & \operatorname{Re}U
\end{bmatrix}.
\]
This yields a structure-preserving reduced-order model that conserves the quadratic energy and
inherits the stability properties of the full system.

\noindent
In this experiment we do  not sample $(q,p)$ from a hypercube.
Instead, we assemble initial conditions from single sine modes with unit amplitude. Let
\[
B := 2,\qquad
\mathcal N := \{1,\dots,B\}.
\]
On the Dirichlet grid $\xi_i = i h$ with $h = L/(N+1)$ ($i=1,\dots,N$), define the discrete
sine vectors
\[
\Phi_n := \left(\sin(n\pi \xi_i/L)\right)_{i=1}^N \in \R^N,\qquad n\in\mathcal N.
\]
We introduce an enumeration $\{(a_j,b_j)\}_{j=1}^{B^2}  := \mathcal N \times \mathcal N$. For each pair we set
\[
q_0^{(j)} = \Phi_{a_j},\qquad p_0^{(j)} = \Phi_{b_j},\qquad
x_0^{(j)} := \begin{bmatrix} q_0^{(j)} \\[2pt] p_0^{(j)} \end{bmatrix}\in\R^{2N},
\quad j=1,\dots,B^2.
\]

\noindent
We assemble the phase-space snapshot matrix
\[
X := \left[x_0^{(1)},\dots,x_0^{(B^2)}\right]\in\R^{2N\times B^2},
\]
and apply the cSVD procedure to $X$ to obtain the symplectic basis
$V\in\R^{2N\times 2n}$ as described above.

\noindent
In the reduced system, we then sample initial conditions directly in the reduced
phase space. More precisely, we draw
\[
z_0 \sim \mathcal{U}\left([{-}z_{\max},z_{\max}]^{2n}\right),
\]
and define the corresponding full initial state as $x_0 = V z_0$.
We retain only those satisfying the energy constraint
\[
Z_{\mathrm{train}}=\left\{z_0^j\in\hat\Omega:\ H(z_0^j)\le H_{\max}\right\}\subset\Omega
\]
with $H_{\max} = 5$ until $N_s = 20000$ states have been collected.

\noindent The training set for the kernel predictor is then given by
\[
\mathcal M^{\Delta T}
    := \left\{\left(I_1 z_0 + I_2\Phi^{\Delta T}(z_0), J_{2n}^{\top} \frac{\Phi^{\Delta T}(z_0)-z_0}{\Delta T}\right) :\ z_0\in Z_{\mathrm{train}} \right\}.
\]
This two-stage procedure (mode-based generation
of full-order initial data, followed by cSVD-based reduction and hypercube
sampling in reduced coordinates) concentrates the basis on
low-frequency wave patterns while preserving flexibility in the reduced
sampling used to train the kernel predictor.

\noindent
For testing we draw $10$ test initial conditions independently of the training
set, using the same mode-based sampling procedure as above. Each test
trajectory is then evolved up to $T=6.0$ using  
(i) the proposed symplectic kernel predictor,  
(ii) the implicit midpoint method with macro time step size $\Delta T$, and  
(iii) a high-fidelity reference solution computed by the implicit midpoint rule
with micro time step $\Delta t$.

\noindent
As a first numerical result, we report in \Cref{fig:fgreedy-convergence-wave}
the convergence behavior of the $f$-greedy center selection.

\begin{figure}[h]
  \centering
  \begin{subfigure}[t]{0.64\textwidth}
    \centering
    \includegraphics[width=\textwidth]{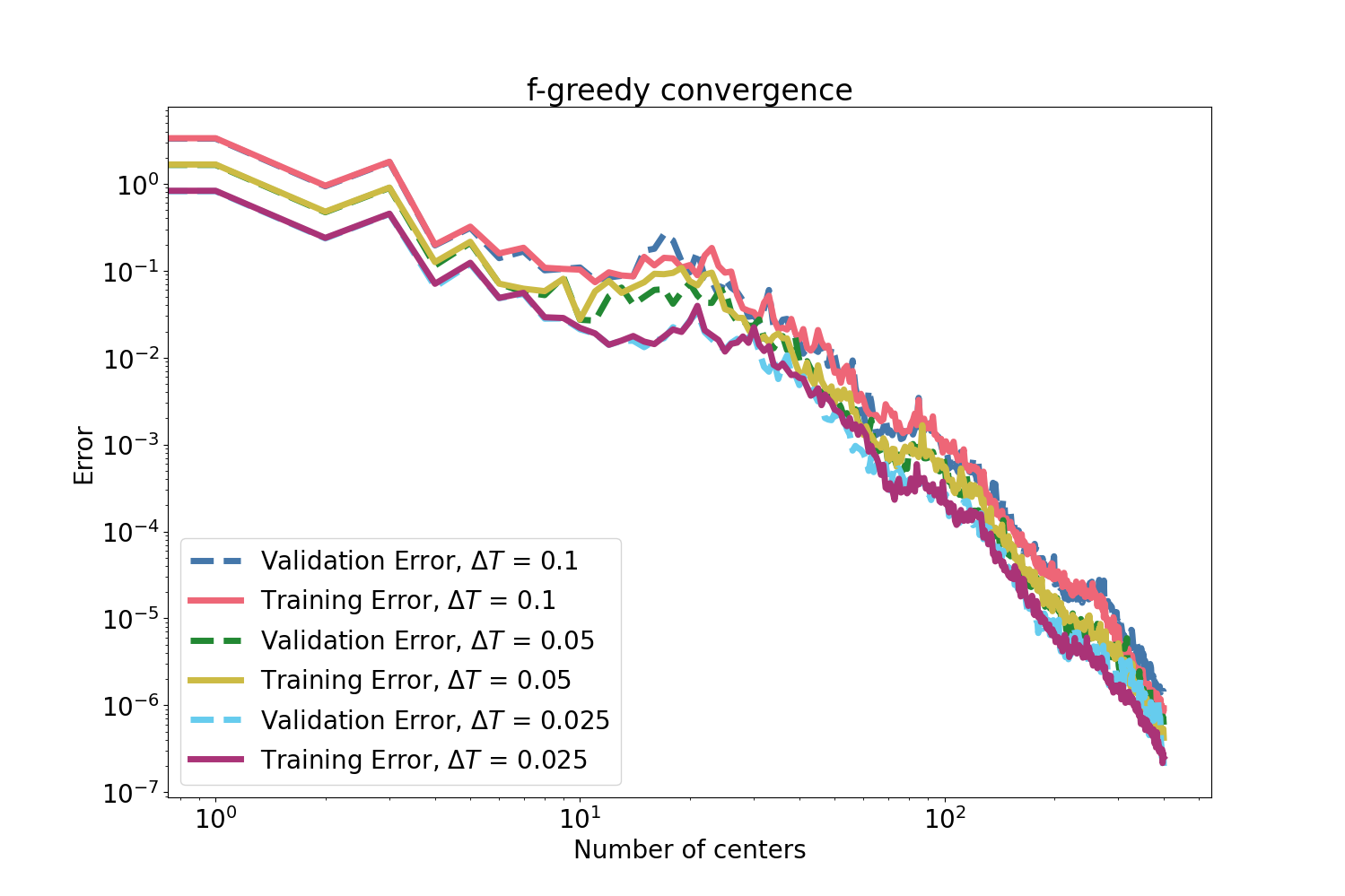}
    \caption{$f$-greedy interpolation error versus the number of selected centers (wave model), training (solid) and validation (dashed).}
    \label{fig:fgreedy-convergence-wave}
  \end{subfigure}\hfill
  \begin{subfigure}[t]{0.35\textwidth}
    \centering
    \includegraphics[width=\textwidth]{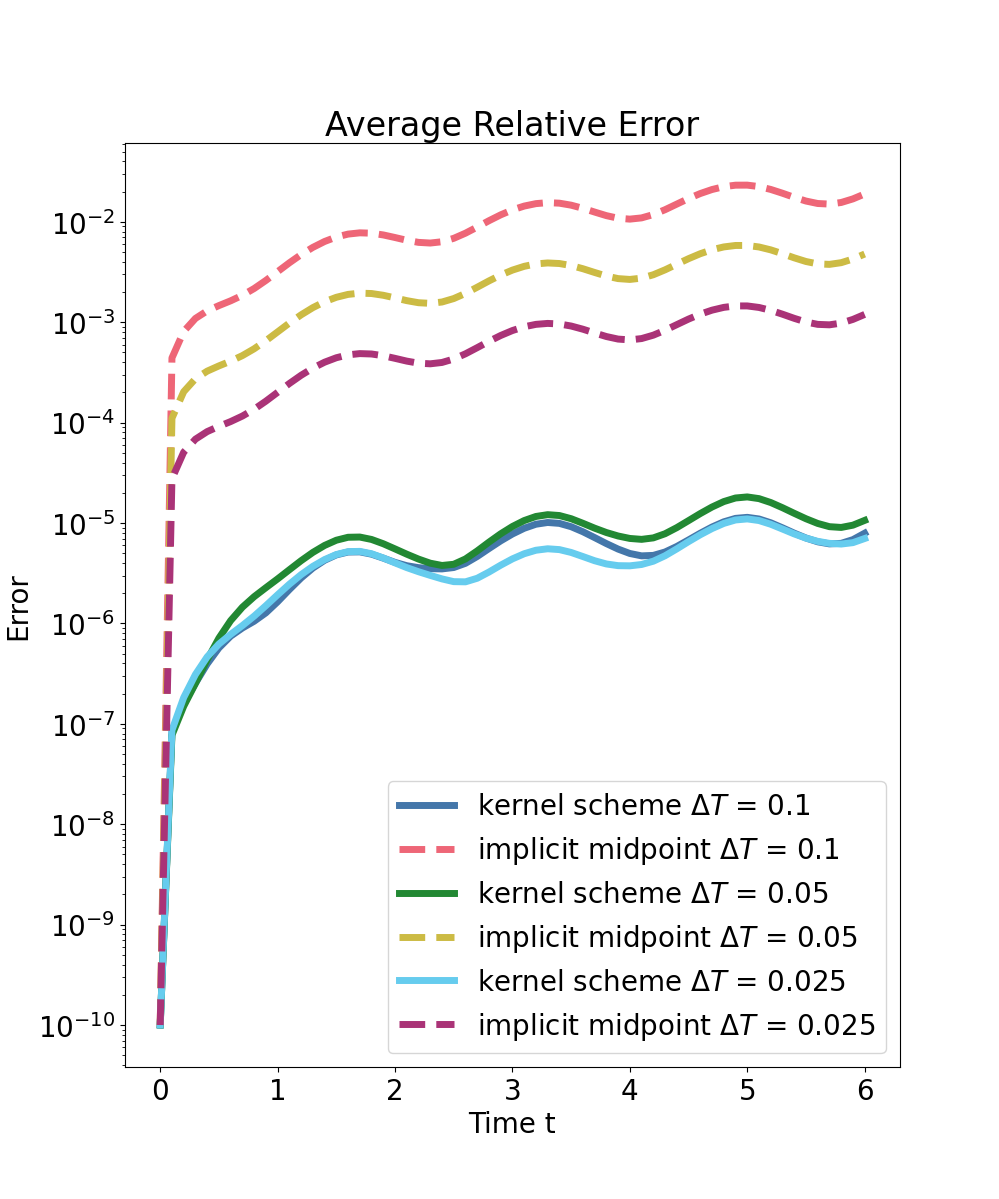}
    \caption{Average relative error over time (wave model), kernel predictor (solid) vs.\ implicit midpoint (dashed).}
    \label{fig:rel-reduction-error-wave}
  \end{subfigure}
  \caption{Reduced 1D wave model: (a) $f$-greedy convergence vs.\ centers; (b) average relative error over time.}
\end{figure}

\noindent In \Cref{fig:fgreedy-convergence-wave} we report the maximum residual error as a
function of the number of centers for $\Delta T\in\{0.1,0.05,0.025\}$. For all
three macro time steps, the curves exhibit a short pre-asymptotic regime for small
numbers of centers (up to roughly $m\approx 10$), followed by an almost
straight-line decay on the log--log scale. This indicates algebraic convergence of
the interpolation error over nearly three orders of magnitude: starting from
values of order $10^0$–$10^1$ at $m=1$, the error drops below $10^{-3}$ after
about $m\approx 30$ centers and reaches values close to $10^{-6}$ by
$m\approx 300-400$. Training and validation errors remain very close for all
$\Delta T$, with the validation curves lying slightly above the training
curves, indicating good generalization and no overfitting. As
expected, smaller macro time steps yield somewhat lower error levels across the
entire range of $m$, but the qualitative decay behavior is essentially the
same for all three step sizes.

\noindent
In \Cref{fig:rel-reduction-error-wave} we show the average relative error for the ten test trajectories over the time interval $[0,6]$. The kernel
predictor achieves errors in the range $10^{-7}$–$10^{-5}$: after a rapid
initial increase from machine precision, the error levels off and exhibits
small temporal oscillations, characteristic of symplectic updates for
oscillatory systems. The three kernel curves for
$\Delta T\in\{0.1,0.05,0.025\}$ nearly coincide, indicating that the method is
robust with respect to the macro time step size in this reduced setting. In
contrast, the implicit midpoint baseline with the same macro time step experiences
significantly larger errors, ranging from about $10^{-3}$ up to roughly
$10^{-1}$ by $t=6$, with a clear ordering:
$\Delta T=0.1$ (largest error), $\Delta T=0.05$, and $\Delta T=0.025$
(smallest error). Thus, for every tested macro time step the kernel predictor
reduces the error by roughly three orders of magnitude compared to the
structure-preserving baseline, while maintaining stable, oscillatory error
profiles over the entire time horizon.

\section{Conclusion and Outlook}\label{Sec:Conclusion}
In this work we proposed a kernel-based surrogate model for Hamiltonian dynamics that is
symplectic by construction and tailored to large time horizons. The core idea is to learn a
scalar potential $s$ whose gradient, evaluated on a mixed argument $(q_0,p_{\Delta T})$, enters an
implicit symplectic-Euler update. This yields a discrete flow map that preserves the canonical
symplectic structure exactly and inherits the favorable long-time
behavior of symplectic integrators. By formulating the learning task as a HB
interpolation problem for gradients, we can leverage the RKHS framework to construct surrogates
with rigorous approximation guarantees and controlled complexity.

\noindent On the theoretical side, we established existence results for the target potential underlying the
symplectic Euler mixed argument. For general Hamiltonian systems we derived conditions on the exact flow
$\Phi^{\Delta T}$ under which there exists a generating function $S^{\Delta T}(q,P)$ whose gradients reproduce
the discrete flow increment on a mixed $(q_0,p_{\Delta T})$ chart. These conditions can be verified
locally under a step-size restriction on $\Delta T$ based on uniform invertibility of the
$(q,p)\mapsto(q,P)$ mapping on compact forward-invariant sets, and globally for quadratic
Hamiltonians outside a discrete resonance set. In addition, we extended existing convergence
theory for greedy kernel approximation to the gradient HB setting, proving
residual bounds.

\noindent On the algorithmic side, we combined the proposed symplectic kernel predictor with model order reduction
(via the cSVD) to obtain a low-dimensional canonical surrogate model for a
high-dimensional discretized PDE. This produces reduced models that remain Hamiltonian and
energy-conserving, while significantly lowering the cost of both training and online evaluation of
the kernel surrogate. In all components, the construction preserves the underlying geometric
structure of the original system.

\noindent The numerical experiments on three benchmark problems---the mathematical pendulum, a nonlinear
spring--mass chain with fixed ends, and a discretized wave equation with symplectic MOR---confirm
the effectiveness of the proposed approach. Across all test cases, the greedy procedure
produces sparse surrogates with nearly algebraic decay of the interpolation error over several
orders of magnitude, and training/validation errors remain closely aligned, indicating good
generalization. In long-time simulations, the symplectic kernel predictor consistently outperforms
a structure-preserving baseline (implicit midpoint with the same macro step), typically reducing
trajectory errors by two to three orders of magnitude while maintaining bounded, oscillatory error
profiles characteristic of symplectic schemes, even in generalization settings where parts of
phase space are unseen during training.

\noindent The present work focuses on learning a time–$\Delta T$ flow map for a fixed macro time
step, and on kernel models with a fixed prediction horizon. An interesting direction for future
research is to relax this constraint and treat the macro step size as an additional variable.
One option is to learn a family of generating functions $S^{\Delta T}$ parameterized by
$\Delta T$, or to augment the input with a time parameter and equip it with a suitable kernel,
thereby obtaining a single surrogate that can be evaluated for a continuum of macro step sizes.
Similarly, the prediction horizon $T$ of the kernel model can be promoted to an explicit parameter,
allowing the construction of surrogates that learn the joint dependence on state and time and
thus support variable-step or adaptive-in-time structure-preserving prediction. We expect such
extensions to broaden the applicability of the proposed framework, in particular for problems
where accuracy and efficiency requirements naturally call for adaptive macro time stepping.

\section*{Funding}
Funded by Deutsche Forschungsgemeinschaft (DFG,
German Research Foundation) under Germany´s
Excellence Strategy – EXC 2075/2 – 390740016.

\section*{Competing interests}
The authors report there are no competing interests to declare.

\section*{Appendix}

First, we focus on how the point-wise derivative error can be bounded by the derivative power function and the error in
the RKHS norm.

\begin{lemma}[Power-function control via \eqref{eqn:rep_deriv}]\label{lem:pf}
For every $m\ge 0$, $x\in \Omega$, and $\ell\in\mathcal J$,
\[
|\partial_\ell e_m(x)|
 = \big|\ip{e_m}{(I-\Pi_m) \partial_\ell^{(2)} k(\cdot,x)}\big|
 \le \norm{e_m} P_m(x,\ell).
\]
Consequently, $\max_{\ell\in\mathcal J}\sup_{x\in \Omega}|\partial_\ell e_m(x)|\le \norm{e_m} \sup_{x,\ell}P_m(x,\ell)$.
\end{lemma}

\begin{proof}
By the reproducing property \eqref{eqn:rep_deriv},
\[
\partial_\ell e_m(x)=\ip{e_m}{\partial_\ell^{(2)} k(\cdot,x)}.
\]
Let $V_m=\mathrm{span}\{\partial^{(2)}_{\ell_i}k(\cdot,x_i):i=1,\dots,m\}$ and let $\Pi_m$ be the orthogonal projector onto $V_m$. Since $s_m=\Pi_m u$, we have $e_m=u-s_m=(I-\Pi_m)u$, hence $e_m\perp V_m$ and in particular
\[
\ip{e_m}{\Pi_m \partial_\ell^{(2)} k(\cdot,x)}=0.
\]
Therefore,
\[
\partial_\ell e_m(x)
=\ip{e_m}{\partial_\ell^{(2)} k(\cdot,x)}
=\ip{e_m}{(I-\Pi_m) \partial_\ell^{(2)} k(\cdot,x)}.
\]
Taking absolute values and using Cauchy--Schwarz,
\[
|\partial_\ell e_m(x)|
=\big|\ip{e_m}{(I-\Pi_m) \partial_\ell^{(2)} k(\cdot,x)}\big|
\le \norm{e_m} \big\|(I-\Pi_m) \partial_\ell^{(2)} k(\cdot,x)\big\|_{\Hk}.
\]
By definition of the power function,
$P_m(x,\ell):=\big\|(I-\Pi_m) \partial_\ell^{(2)} k(\cdot,x)\big\|_{\Hk}$,
which gives the claimed inequality. Maximizing over $x\in \Omega$ and $\ell\in\mathcal J$ yields
\[
\max_{\ell\in\mathcal J}\sup_{x\in \Omega}|\partial_\ell e_m(x)|
\le \norm{e_m} \sup_{\ell\in\mathcal J}\sup_{x\in \Omega} P_m(x,\ell).
\]
\end{proof}

\noindent Next, we focus on quantifying the error decay in the RKHS norm. 

\begin{lemma}[Orthogonal update and projection error]\label{lem:orth}
Let
\begin{align}
a_m &:=  |\partial_{\ell_{m+1}} e_m(x_{m+1})| =
\big|\ip{e_m}{(I-\Pi_m)\partial^{(2)}_{\ell_{m+1}}k(\cdot,x_{m+1})}\big|,\\
b_m &:= P_m(x_{m+1},\ell_{m+1})
=\big\|(I-\Pi_m)\partial^{(2)}_{\ell_{m+1}}k(\cdot,x_{m+1})\big\|.
\end{align}
Then
\begin{align*}
s_{m+1} &= s_m+\frac{\ip{e_m}{(I-\Pi_m)\partial^{(2)}_{\ell_{m+1}}k(\cdot,x_{m+1})}}{b_m^{ 2}} 
\left((I-\Pi_m)\partial^{(2)}_{\ell_{m+1}}k(\cdot,x_{m+1})\right),\\
\norm{e_{m+1}}^{2}\ &=\ \norm{e_m}^{2}-\frac{a_m^{2}}{b_m^{2}}.
\end{align*}
\end{lemma}
\noindent (We adopt the convention that if $b_i=0$ which implies $a_i=0$, the fraction is $0$.)
\begin{proof}
Recall $V_m=\operatorname{span}\{\partial^{(2)}_{\ell_i}k(\cdot,x_i):i=1,\dots,m\}$ and that
$s_m=\Pi_m u$ is the orthogonal projector of $u$ onto $V_m$, so $e_m=u-s_m=(I-\Pi_m)u$ and
$e_m\perp V_m$. Put
\[
w_m\ :=\ (I-\Pi_m) \partial^{(2)}_{\ell_{m+1}}k(\cdot,x_{m+1}) \in \Hk.
\]
Since $\Pi_m$ is an orthogonal projector ($\Pi_m=\Pi_m^*$ and $\Pi_m^2=\Pi_m$), for any
$v\in V_m$ we have $\Pi_m v=v$ and hence
\begin{align*}
\ip{w_m}{v}
&=\ip{\partial^{(2)}_{\ell_{m+1}}k(\cdot,x_{m+1})-\Pi_m\partial^{(2)}_{\ell_{m+1}}k(\cdot,x_{m+1})}{v} \\
&=\ip{\partial^{(2)}_{\ell_{m+1}}k(\cdot,x_{m+1})}{v}-\ip{\partial^{(2)}_{\ell_{m+1}}k(\cdot,x_{m+1})}{\Pi_m v}\\
&=0.
\end{align*}
Thus $w_m\perp V_m$ and therefore
\[
V_{m+1}\ =\ V_m\oplus \operatorname{span}\{w_m\}
\]
is an orthogonal direct sum.

\noindent If $b_m=\norm{w_m}=0$, then $w_m=0$, so $V_{m+1}=V_m$, $s_{m+1}=s_m$ and $e_{m+1}=e_m$. In this case
$a_m=|\ip{e_m}{w_m}|=0$, and both displayed identities hold trivially. Hence we may assume $b_m>0$
in the remainder.

\noindent Because $V_{m+1}$ is an orthogonal sum and $s_{m+1}$ is the orthogonal projection of $u$ onto
$V_{m+1}$, the standard formula for projection onto a one-dimensional orthogonal extension gives
\[
s_{m+1}\ =\ s_m+\alpha_m  w_m,
\qquad
\alpha_m\ :=\ \frac{\ip{u-s_m}{w_m}}{\norm{w_m}^2}
\ =\ \frac{\ip{e_m}{w_m}}{b_m^2}.
\]
This yields the asserted update formula for $s_{m+1}$.

\noindent For the projection error identity, compute
\[
e_{m+1}\ =\ u-s_{m+1}\ =\ u-\left(s_m+\alpha_m w_m\right)\ =\ e_m-\alpha_m w_m.
\]
Hence
\[
\begin{aligned}
\norm{e_{m+1}}^{2}
&= \norm{e_m-\alpha_m w_m}^{2}
= \norm{e_m}^{2}-2\alpha_m \ip{e_m}{w_m}+\alpha_m^{2} \norm{w_m}^{2} \\
&= \norm{e_m}^{2}-2 \frac{\ip{e_m}{w_m}}{\norm{w_m}^2} \ip{e_m}{w_m}
   +\left(\frac{\ip{e_m}{w_m}}{\norm{w_m}^2}\right)^{2} \norm{w_m}^{2} \\
&= \norm{e_m}^{2}-\frac{|\ip{e_m}{w_m}|^{2}}{\norm{w_m}^{2}}
=\ \norm{e_m}^{2}-\frac{a_m^{2}}{b_m^{2}},
\end{aligned}
\]
because $a_m=|\ip{e_m}{w_m}|$ by definition and
$b_m=\norm{w_m}$, i.e., $\norm{e_{m}}^2$ is nonincreasing.
\end{proof}
\noindent In the next lemma, we bound the geometric mean of the quotients
$|\partial_{\ell_{m+1}} e_m(x_{m+1})|/P_m(x_{m+1},\ell_{m+1})$ over a block.

\begin{lemma}\label{lem:amgm}
For any $m\ge 1$,
\[
\left[\ \prod_{i=m+1}^{2m}\frac{a_i}{b_i}\ \right]^{1/m}\ \le\ m^{-1/2}  \|e_{m+1}\|_{\Hk}.
\]

\end{lemma}

\begin{proof}
From Lemma~\ref{lem:orth}, for every $i\ge 0$,
\[
\|e_{i+1}\|_{\Hk}^{2}=\|e_i\|_{\Hk}^{2}-\frac{a_i^{2}}{b_i^{2}}.
\]
Set $t_i:=a_i^{2}/b_i^{2}\ge 0$ for $i=m+1,\dots,2m$. Summing gives
\[
\sum_{i=m+1}^{2m} t_i
=\|e_{m+1}\|_{\Hk}^{2}-\|e_{2m+1}\|_{\Hk}^{2}
\ \le\ \|e_{m+1}\|_{\Hk}^{2}.
\]
By the Arithmetic mean--geometric mean (AM--GM) inequality applied to $t_{m+1},\dots,t_{2m}$,
\[
\left(\ \prod_{i=m+1}^{2m} t_i\ \right)^{1/m}
\le \frac{1}{m}\sum_{i=m+1}^{2m} t_i
\le \frac{\|e_{m+1}\|_{\Hk}^{2}}{m}.
\]
Taking square roots yields
\[
\left(\ \prod_{i=m+1}^{2m} \frac{a_i^{2}}{b_i^{2}}\ \right)^{1/(2m)}
\ \le\ m^{-1/2}  \|e_{m+1}\|_{\Hk}.
\]
Since
\(
\left(\prod_{i=m+1}^{2m} \frac{a_i^{2}}{b_i^{2}}\right)^{1/(2m)}
=\left[\prod_{i=m+1}^{2m}\frac{a_i}{b_i}\right]^{1/m},
\)
the claim follows. If some $b_i=0$, then $a_i=0$ and the product on the left is $0$,
so the inequality is trivial.
\end{proof}

\noindent Next, we combine the three lemmas to finally prove \Cref{thm:HB-mean} and bound the geometric mean of the maximum derivative errors,
i.e., of $\max_{\ell\in\mathcal J}\ \sup_{x\in \Omega} |\partial_\ell e_i(x)|$.

\begin{proof}[Proof of \hypertarget{proof:HB-mean}{Theorem \ref{thm:HB-mean}}]
Fix $m\ge 1$ and define $a_i, b_i$ as in Lemma~\ref{lem:orth}. \\ 
\noindent By Lemma~\ref{lem:orth}, $\|e_{i+1}\|_{\Hk}^2=\|e_i\|_{\Hk}^2-a_i^2/b_i^2\le \|e_i\|_{\Hk}^2$,
hence $\|e_i\|_{\Hk}\le \|e_{m+1}\|_{\Hk}$ for all $i\in\{m+1,\dots,2m\}$.
Lemma~\ref{lem:amgm} gives
\begin{equation}\label{eq:block-amgm}
\Bigg(\ \prod_{i=m+1}^{2m}\frac{a_i}{b_i}\ \Bigg)^{1/m}\ \le\ m^{-1/2}  \|e_{m+1}\|_{\Hk}.
\end{equation}
For each $i\in\{m+1,\dots,2m\}$, the $f$-greedy rule picks $(x_{i+1},\ell_{i+1})$
with
\[
a_i=|\partial_{\ell_{i+1}} e_i(x_{i+1})|
=\max_{\ell\in\mathcal J}\ \sup_{x\in\Omega} |\partial_\ell e_i(x)|.
\]
Thus
\begin{align}
\Bigg(\prod_{i=m+1}^{2m}\max_{\ell \in \mathcal J}\sup_x |\partial_\ell e_i(x)|\Bigg)^{1/m}
&=\Bigg(\prod_{i=m+1}^{2m} a_i\Bigg)^{1/m}
=\Bigg(\prod_{i=m+1}^{2m} \frac{a_i}{b_i}\Bigg)^{1/m}
   \Bigg(\prod_{i=m+1}^{2m} b_i\Bigg)^{1/m} \notag \\ 
   &\le\ m^{-1/2} \|e_{m+1}\|_{H_k(\Omega)} 
\left[\ \prod_{i=m+1}^{2m} P_i(x_{i+1},\ell_{i+1})\ \right]^{1/m}.\label{eq:HB-mean-res-error}
\end{align}

\noindent Let $u_i := \max_{\ell \in \mathcal J}\sup_{x\in\Omega}|\partial_\ell e_i(x)| \ge 0$. Then
\[
\min_{m+1\le i\le 2m} u_i
= \Bigg(\prod_{i=m+1}^{2m} \min_{m+1\le i\le 2m} u_i\Bigg)^{ 1/m}
\le \Bigg(\prod_{i=m+1}^{2m} u_i\Bigg)^{ 1/m},
\]
and thus, by \eqref{eq:HB-mean-res-error},
\[
\min_{m+1\le i\le 2m}\ \max_{\ell \in \mathcal J}\ \sup_{x\in\Omega}|\partial_\ell e_i(x)|
\ \le\ m^{-1/2} \|e_{m+1}\|_{H_k(\Omega)} 
\left[\ \prod_{i=m+1}^{2m} P_i(x_{i+1},\ell_{i+1})\ \right]^{1/m}.
\]

\noindent
Using
\[
\|\nabla e_i(x)\|_2 \le \sqrt n\ \max_{\ell \in \mathcal J}|\partial_\ell e_i(x)| \quad\text{for all } x\in\Omega
\]
yields
\begin{equation}\label{eq:grad-block}
\min_{m+1\le i\le 2m}\ \|\nabla e_i\|_{L^\infty(\Omega)}
\ \le\ \sqrt{n}  m^{-1/2} \|e_{m+1}\|_{H_k(\Omega)} \left[\ \prod_{i=m+1}^{2m} P_i(x_{i+1},\ell_{i+1})\ \right]^{1/m}.
\end{equation}

\end{proof}

\end{document}